\documentclass[12pt,reqno]{amsart}

\usepackage{amsmath,amsthm,amssymb,amscd,euscript,longtable,enumitem}

\usepackage[
  style=numeric-comp,
  sorting=nyt,
  giveninits=true,
  maxnames=4,
  minnames=1,
  backend=biber,
  doi=false,
  isbn=false,
  url=false,
  eprint=true
]{biblatex}
\AtEveryBibitem{
	\clearfield{issn} 
	\clearfield{doi} 
	\clearfield{isbn} 
	
	\ifentrytype{online}{}{
		\clearfield{url}
	}
}
\usepackage{graphicx,color,mathtools,tikz}
\usepackage[colorlinks=true]{hyperref}
\hypersetup{
    linkcolor=black,
    citecolor=red,
    urlcolor=blue
}

\usepackage{nicematrix}

\usetikzlibrary{math,fit} 

\usepackage[margin=2.5cm]{geometry}
\def \r{\mathbb R}

\def \q{\mathbb Q}
\def \z{\mathbb Z}
\def \n{\mathbb N}

\def \M{\mathfrak M}

\DeclareMathOperator{\polyomino}{polyomino}
\DeclareMathOperator{\Aff}{Aff}
\DeclareMathOperator{\Al}{Al}

\DeclareMathOperator{\Head}{Head}
\DeclareMathOperator{\LC}{LC}
\DeclareMathOperator{\LLS}{LLS}

\DeclareMathOperator{\SL}{SL}
\DeclareMathOperator{\GLC}{GLC}
\DeclareMathOperator{\GL}{GL}
\DeclareMathOperator{\Mat}{Mat}

\DeclareMathOperator{\tr}{tr}

\renewbibmacro{in:}{}

\makeatletter
\def\keywords{\xdef\@thefnmark{}\@footnotetext}
\makeatother


\usepackage[todonotes={textsize=footnotesize}]{changes}
\definechangesauthor[color=blue]{OK}
\definechangesauthor[color=red]{MA}

\DeclareMathOperator{\wug}{wug}

\newtheorem{theorem}{Theorem}[section]

\newtheorem{proposition}[theorem]{Proposition}
\newtheorem{corollary}[theorem]{Corollary}
\theoremstyle{remark}
\newtheorem{remark}[theorem]{Remark}
\theoremstyle{definition}
\newtheorem{definition}[theorem]{Definition}
\newtheorem{example}[theorem]{Example}
\newtheorem{problem}{Problem}

\author[O.~Karpenkov]{Oleg~Karpenkov}
\address{Department of Mathematical Sciences\\ University of Liverpool\\ Peach Street \\ Liverpool L69~7ZL}
\email{karpenkov@liv.ac.uk}

\title{Wug-Snake Graphs and Markov Numbers of Matrix Semigroups}
\thanks{The second author is supported by a CSC grant from the Chinese government.}
\author[Y.~Ma]{Yefei~Ma}
\address{IMAG – UMR 5149\\ Universit\'e de Montpellier\\ France}
\email{yefei.ma@umontpellier.fr}

\date{\today}

\begin{document}
\begin{abstract}
Classically, Markov numbers are recovered as perfect matching numbers of
domino snake graphs. We extend this correspondence by introducing weighted
universal generalised snake graphs, or wug-snake graphs. These are weighted
ordered bipartite graphs whose perfect matching sequences encode linear
recurrences. To every wug-snake graph we associate a continuant matrix and
prove that the determinant of this matrix equals the weighted perfect matching sum.

We then introduce polyomino wug-tiles, bodies of wug-snake graphs that act
linearly on state vectors. Every integer matrix admits a canonical
polyomino wug-tile. Our main result identifies the wug-snake determinant of
a tile representing a matrix $A$ with the Markov-Davenport form of $A$.
Consequently, algebraic and geometric Markov numbers of matrices and matrix
semigroups can be expressed as weighted perfect matching determinants. Further we define Frobenius maps for matrix semigroups and discuss examples
recovering classical Markov numbers
%, weighted PLLS-sequences, 
and higher-dimensional lattice realisations.
\end{abstract}

\maketitle

\tableofcontents

\section{Introduction}

%The classical theory of Markov numbers links several apparently different
%objects: minima of indefinite quadratic forms, continued fractions, Cohn
%matrices, Farey arithmetic, and perfect matchings of special snake graphs.
%The aim of this paper is to extend this circle of ideas from the classical
%discrete Markov spectrum to a broader matrix-semigroup setting.

\subsection{Background and motivation}
In 1879, A.~Markov studied the minimal nonzero absolute values of indefinite
binary quadratic forms at integer points. In~\cite{Markoff1879,Markoff1880}
he introduced the spectrum of such minima, now called the
\textit{Markov spectrum}. Its subset below $3$ is
discrete; see
e.g.~\cite{Cusick1989}. This subset gives rise to the classical Markov numbers, whose set we denote by $\mathcal M$. These are the
positive integers appearing in solutions of the Markov equation
\[
    x^2+y^2+z^2=3xyz.
\]
Via the {\it Frobenius map} $\phi: \q_{0,1} \to \mathcal M$, $t \mapsto m_t$, one can index a Markov number $m_t$ by a rational number $t$, which we call the {\it Farey index} of $m_t$. In 1955, H.~Cohn linked $m_t$ to special matrices in $\SL(2,\z)$, now called Cohn
matrices~\cite{Cohn1955}. This multiplicative viewpoint enables one to express the Farey indexed Markov number $m_t$ via the perfect matching number of a so-called ``domino snake graph'' $D(t)$; see Subsection \ref{subsec:domino-snake-perfect-matchings}. We denote this perfect matching number by $\mu(D(t))$. A celebrated theorem, see \cite[Theorem 7.12]{Aigner2013} states the elegant formula $\mu(D(t)) = m_t$ for every $t \in \q_{0,1}$.

On the other hand, a multidimensional version of the Markov spectrum was introduced by
H.~Davenport in his study of cubic forms; he computed the first few Markov
minima in dimension three~\cite{Davenport1941,Davenport1943,Davenport1947}. 

\vspace{2mm}

The above historical theories (see a more complete description in Section \ref{Markov minima and Markov spectrum} and \ref{Basic properties of discrete Markov spectrum}) lead to the following natural problem.

\vspace{2mm}

\noindent
\textbf{Guiding problem.}
{\it Develop a perfect-matching theory that extends the classical domino snake graph description of Markov numbers beyond the Cohn matrices, and that applies
to integer matrices, matrix semigroups, and higher-dimensional
Markov-Davenport forms.}

\vspace{2mm}

In this paper, we establish a unifying perfect-matching theory that resolves this problem.

\subsection{Main results}

The classical theory relates Markov numbers to three apparently different
objects: Cohn matrices, perfect matchings of
domino snake graphs (mentioned above), and continued fractions (see Section \ref{subsection: Discrete Markov spectrum and Markov equation}). In this paper we show that this relationship is a
special case of a more general mechanism connecting weighted perfect
matchings, linear recurrences, matrix semigroups, and Markov-Davenport
forms.

\medskip

First, for $n \in \z_{\ge 0} \cup \{\infty\}$, we introduce \textit{wug-snake graphs} $K_n(w)$. These are ordered
weighted bipartite graphs whose weight matrices $w = (w_{ij})$ are super upper triangular (Definition \ref{def: super upper triangular})
and whose subdiagonal entries are equal to 1. In this construction, we show that, in particular, every domino snake graph can be represented as a wug-snake graph (Proposition \ref{prop: domino are wug-snake}). The first key result following this construction is as follows.

\medskip

\noindent
\textbf{A. Wug-snake graphs and continuants.}
To every wug-snake graph $K_n(w)$ we associate a continuant matrix
$C_n(w)$. We prove that the weighted perfect matching sum $\mu(K_n(w))$ is given by the determinant of the continuant matrix
$C_n(w)$. Moreover, the perfect matching sequence satisfies a recurrence formula
\[
    \mu(K_n(w))
    =
    \sum_{i=1}^n w_{i,n}\,\mu(K_{i-1}(w)).
\]
See
Theorem~\ref{theorem:wug-det}.

\medskip

Next, in order to finely describe the structure of a wug-snake graph $K_n(w)$, we decompose it into two parts: an initial segment called the \textit{head} $H$, and the remainder called the \textit{body} $B$, where the set of edges connecting vertices of the head to vertices of the body is called the \textit{neck}. An important numerical invariant is called the \textit{rank} of $K_n(w)$, which is defined by
\[
        \max\{j-i+1 : w_{i,j}\neq 0,\; i\leq j\}.
\]

With all these notions prepared, one can take the product $B_1B_2$ of two compatible bodies $B_1$ and $B_2$ to enlarge the graph. Given $K_n(w)$ which decomposes as $HB$ with head length $d$, for every $k\geq 0$, we define the $k$-th state vector 
\[
v_k(HB) \in \z^d,
\]
which is the last $d$ entries of the perfect matching sequence of $HB^k$; see Definition \ref{def:sum-wugs}. This brings us to three important notions:
\begin{enumerate}
    \item Given a wug-snake graph $K_n(w)$, decomposed as $HB$, whose head $H$ has length $d$. The \textit{wug-snake determinant} is defined by
\[
   \det(K_n(w)) = \det\bigl(v_0(HB),v_1(HB),\ldots,v_{d-1}(HB)\bigr);
\]
see Definition \ref{Markov-MDavenport-det}.

\item Let \(B\) be a body whose neck-length does not exceed \(d\). For
\(x\in\mathbb Z^d\), choose a head \(H_x\) of length \(d\) with state
vector \(v_0(H_xB)=x\). For any $k \geq 0$, define the \textit{$k$-step snake operator} of $B$ as
\[
   \mathcal S_{B^k}^d: \z^d \to \z^d,\quad x \longmapsto v_k(H_xB)
\]
which is linear and whose matrix is denoted by $M^d_{B^k}$; see Definition \ref{def: snake operator}.
\item The \textit{polyomino wug-tile} for $A \in \Mat(d,\z)$, which is a body $B$ whose 1-step snake operator satisfies $M_B^d = A$; see Definition \ref{def: Polyomino wug-tile}.
\end{enumerate}
Given a sequence of integers $(v_1,\ldots,v_m)$ of length $m$ and a matrix $A \in \Mat(m,\z)$, we next give a canonical construction of a wug-snake graph, whose head and body encode the two data, respectively; see Subsection \ref{subsection: Canonical construction}. The canonical construction has rank at most $2m-1$. 
Under a natural
nondegeneracy condition, any representation has rank at least $m$;
see Propositions~\ref{rank-canonic} and Theorem~\ref{transFrob-tiles}$($iii$)$. 
So, it is natural to ask if it is possible to give an explicit construction of a wug-snake graph of the minimal possible rank. We provide such constructions in Item~\textbf{C}.

\vspace{2mm}

The second key result is as follows.
\medskip

\noindent
\textbf{B. Markov numbers as wug-snake determinants.} The key to link the wug-snake determinants with Markov numbers is through Markov-Davenport forms: given $A \in \Mat(d,\z)$ and a sequence $x= (x_1,\ldots,x_d)$, let $B$ and $H$ be any polyomino wug-tile for $A$ and any head whose perfect matching sequence is $x$, i.e. $v_0(HB)=x$ (one can take, for instance, the canonical construction mentioned above). Then,
\[
f_A(x_1,\ldots,x_d) = \det(HB),
\]
where $f_A$ is the Markov-Davenport form of $A$ and $\det(HB)$ is the wug-snake determinant, i.e. the determinant
formed from the state vectors of $H,HB,\ldots,HB^{d-1}$; see Theorem \ref{MD-det}. An immediate consequence is that the \textit{geometric Markov number} of $A$, defined by $m(A)= \inf_{p \in \mathbb{Z}^n \setminus \{0\}} |f_A(p)|$, (see Definition \ref{alg-geom}, which is a multidimensional analogue of the classical Markov minimum) can be found by only searching for a head of the wug-snake graph. More precisely (see Corollary \ref{alg-geo-Markov-det-wug}(i)), there exists a head $H$ of length $d$ such that
    \[
        m(A)=|\det(HB)|.
    \]
Another key matrix invariant, which we call \textit{algebraic Markov number} (see Definition \ref{alg-geom}), is motivated by the classic Cohn matrix where Markov number appears as the upper right entry of a matrix in $\SL(2,\z)$. In the multidimensional setting, we define algebraic Markov number of $A$ by $\hat m(A) = |f_A(0,\ldots,0,1)|.$ Returning to the wug-snake setting (see Corollary \ref{alg-geo-Markov-det-wug}(ii)), if $H_0$ is the canonical realisation of the sequence $(0,\ldots,0,1)$, then the algebraic Markov number of $A$ is
\[
\hat m(A) = |\det (H_0B)|.
\]
In general, algebraic and geometric Markov number may not coincide. An interesting problem of determining for which $A \in \Mat(d,\z)$ the two Markov numbers coincide is not fully considered in this paper. We provide one such case in Example \ref{exM111M1015}.

\medskip
\noindent
\textbf{C. Construction of minimal rank wug-snake graphs.} In Theorem~\ref{thm:lc-mat-glc-ll} we prove that for $m \geq 2$ every matrix in $\Mat(m,\z)$ can be written as a product of lower companion integer matrices.
These matrices encode information from \textit{linear recurrence system}; see Definition~\ref{Frob-to-E}. For $N \in \n$, given an integer sequence $(x_1,x_2,\ldots,x_N)$ defined by a linear recurrence system, we construct a wug-snake graph (see Definition \ref{Perfect matching of a sequence -- definition}) so that the perfect matching sequence is exactly this prescribed integer sequence; see Theorem \ref{Perfect matching of a sequence -- theorem}. Following this construction, we show that every matrix in $\Mat(m,\z)$ under some non-degeneracy condition admits a polyomino wug-tile of rank $m$; see Theorem~\ref{transFrob-tiles}.

\medskip
\noindent
\textbf{D. Markov numbers for matrix semigroups.} 
Following recent studies on the computation of Markov numbers for various semigroups with two variables~\cite{KvS2020}, we make first steps to describe the cases when the number of generators exceed 2. 
We introduce an \textit{affine Farey-type subdivision scheme} and a \textit{Cohn map} from the vertices of the
associated tessellation to the semigroup; see Subsection~\ref{Markov numbers of semigroups: basic definitions}. Composing this map with the algebraic and geometric Markov number gives us
\textit{Frobenius maps for the semigroup}. This construction recovers the classical
Markov numbers in the Cohn case and produces weighted and
higher-dimensional approaches linking  the wug-snake graphs and their perfect matchings to multidimensional geometry of numbers. Table~\ref{tab:cohn_comparison} below gives one such example.
\vspace{2mm}
\begin{table}[htbp]
\centering % Better practice than \begin{center} as it avoids extra vertical spacing
\begin{tabular}{|c|c|c|}
\hline
&  Cohn matrix &  Non-Cohn matrix \\
\hline
Matrix &  
$
\begin{pmatrix}
41&29\\
65&46\\
\end{pmatrix}
$
& 
$
\begin{pmatrix}
62&43\\
111&77\\
\end{pmatrix}
$
\\
\hline
$
\begin{array}{c}
\hbox{Wug-snake}
\\
\hbox{graph}
\end{array}
$
&
$
\begin{array}{c}
\includegraphics[height=.75cm]{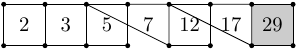}
\end{array}
$
&  
$
\begin{array}{c}
\includegraphics[height=1.1cm]{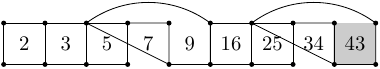}
\end{array}
$
\\
\hline
$
\begin{array}{c}
\hbox{Euclidean}
\\
\hbox{realisation}
\end{array}
$
&
$
\begin{array}{c}
\includegraphics[height=2cm]{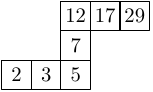}
\end{array}
$
&  
$
\begin{array}{c}
\includegraphics[height=3.5cm]{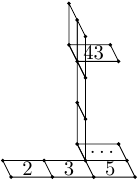}
\end{array} 
$
\\
\hline
\end{tabular}
\vspace{2mm}
\caption{Comparison between Cohn and Non-Cohn matrices and their geometric realisations. On the left, the Cohn matrix represents the Markov number $29$; the corresponding Euclidean realisation is in $\r^2$ with perfect matching number $29$. On the right, the non-Cohn matrix represents the algebraic Markov number $43$; the corresponding Euclidean realisation is in $\r^3$ with perfect matching number~$43$; see Example~\ref{example-3d}.}
\label{tab:cohn_comparison} % Always put \label AFTER \caption
\end{table}

\noindent
\textbf{Further discussions on lattice realisations of wug-snake graphs.} We make the first steps in the study of lattice realisations of wug-snake graphs. In the classical
two-dimensional case, the slope of a snake graph recovers the Farey
coordinate of the corresponding Markov number. We extend this idea to
weighted and higher-dimensional settings, where rational directions are
replaced by multidimensional Farey data and by slowly increasing sequences
of cubes; see Section \ref{Lattice realisation of wug-snake graphs}. These constructions lead to several open questions about additive
structures for higher-dimensional snake-like lattice configurations.

\subsection{Related works}
\subsubsection{Markov numbers}
Markov spectrum is closely related to Diophantine
approximation and to regular continued fractions. Its part below $3$ is discrete, while the part above $3$ has a rich structure of limit points; see
e.g.~\cite{Cusick1989}.
The discrete part gives rise to the classical Markov numbers that are positive integer solutions of the Markov equation. Following the work by F.~Frobenius related~\cite{fro1913}, H.~Cohn linked
Markov numbers to special matrices in $\SL(2,\z)$, now called Cohn matrices~\cite{Cohn1955}.

\vspace{1mm}

A multidimensional version of the Markov spectrum was introduced by
H.~Davenport in his study of cubic forms; he computed the first few Markov
minima in dimension three~\cite{Davenport1941,Davenport1943,Davenport1947}.
The multidimensional theory is also connected to the Littlewood conjecture,
as shown by J.W.S.~Cassels and H.P.F.~Swinnerton-Dyer~\cite{CasselsSwinnertonDyer1955}.

\vspace{2mm}

More recently, Markov numbers have appeared in the theory of quiver
representations and cluster algebras; see, for example, the works of
R.~Schiffler~\cite{Sch2014,Schiffler2024} and the foundational paper of
S.~Fomin and A.~Zelevinsky~\cite{FZ2002}. Generalised Markov numbers from
cluster-algebraic constructions were studied by E.~Banaian and A.~Sen~\cite{BS2024}.
Quantised Markov numbers were introduced by S.~Evans, P.~Jouteur,
S.~Morier-Genoud, and V.~Ovsienko~\cite{EJ2025}. Further related developments
include the connection with Jones polynomials discovered by K.~Lee and
R.~Schiffler~\cite{KS2019}, and the study of generalised Vieta polynomials by
S.~Evans, A.~Veselov, and B.~Winn~\cite{EVW2025}.

\subsubsection{Snake graphs and perfect matchings}
Perfect matchings on square grids have a long history, going back to the
dimer model studied by P.W.~Kasteleyn~\cite{Kasteleyn1961} and by H.N.V.~Temperley and M.E.~Fisher~\cite{FT1961}. In the setting of Markov numbers, the relevant graphs are special snake graphs, or domino
graphs. M.~Aigner~\cite{Aigner2013}, building on ideas of H.~Cohn~\cite{Cohn1991} and J.~Propp~\cite{Propp1999},
explained a remarkable correspondence: the number of perfect matchings of
certain domino snake graphs recovers the classical Markov numbers.

\vspace{2mm}

Snake graphs have since become a standard combinatorial bridge between
continued fractions, cluster algebras, and Markov-type phenomena. An
important method for reading continued fractions from snake graphs was
developed by \.{I}.~\c{C}anak\c{c}\i\ and R.~Schiffler
\cite{CS2013,CS2020}. Perfect matchings of weighted snake graphs were
studied by G.~Musiker~\cite{Mus2008}, and later by  S.~Evans, P.~Jouteur, S.~Morier-Genoud, and V.~Ovsienko~\cite{EJ2025}, and by J.-C.~Aval and
S.~Labb\'e~\cite{AL2026}. Higher dimer covers on snake graphs were investigated
by G.~Musiker et al.~\cite{MOS2026}, and related octagonal snake graphs were
studied by A.~Ciliberti~\cite{AC2025}.

\subsection{Organisation of the paper}
In Section~\ref{Markov minima and Markov spectrum}, we start with the classical
theory of Markov minima, Markov-Davenport forms. We also introduce the algebraic and geometric Markov numbers
of a matrix.

\vspace{1mm}

In Section~\ref{Basic properties of discrete Markov spectrum}, we review
the classical links among Markov numbers, Farey indexing, Cohn matrices, and
perfect matchings of domino snake graphs. This section provides the
classical model that motivates the later constructions.

\vspace{1mm}

In Section~\ref{Generalised universal weighted snake graph}, we define
wug-snake graphs and prove their basic recurrence properties.
In particular we introduce continuant
matrices for wug-snake graphs. 
Further we define
heads, bodies, polyomino wug-tiles, and wug-snake determinants, and we prove
the connection with Markov-Davenport forms. 
After that we show that every matrix of size $m$ is decomposable into a product of lower companion matrices and hence it can be represented by a wug-tile of size $m$. 
Finally we give
a brief discussion of possible extensions to face matchings in
CW-complexes.

\vspace{1mm}

In Section~\ref{Markov numbers of semigroups: definitions and examples}, we
develop the theory of Markov numbers for matrix semigroups. We define
Farey-type tessellations, Cohn maps, algebraic and geometric Frobenius maps,
and Markov semigroups. We then discuss several examples, including the
classical Markov semigroup, semigroups with PLLS-sequences $(a,a)$ and $(b,b)$,
and a semigroup with three generators.

\vspace{1mm}

In Section~\ref{Lattice realisation of wug-snake graphs}, we study lattice
realisations of wug-snake graphs. We discuss slopes in the classical
two-dimensional case, higher-dimensional analogues, and slowly increasing
sequences of cubes.

%\section{Preliminaries}\label{Preliminaries}

%We start in Subsection~\ref{Perfect matchings for bipartite graphs} by recalling the notions of permanents of matrices and perfect matchings of graphs, highlighting a key property involved in counting the perfect matchings of a bipartite graph. Next, in Subsection~\ref{section1.2}, we review a classical theorem connecting Markov numbers to the perfect matchings of domino graphs, introducing all the relevant concepts required to prove this theorem.

%Further, in Subsection~\ref{Matrix representation of continued fractions}, we explore the matrix representation of classical continued fractions, both as a product of $2 \times 2$ matrices and in terms of determinants of $n \times n$ matrices. We recall the notion of reduced $2 \times 2$ integer matrices in Subsection~\ref{Semigroup of reduced GL2Z matrices}. In Subsection~\ref{Matrix representation of recurrence sequences and corresponding subtractive Euclidean algorithms}, we examine the matrix representation of recurrence sequences. Finally, in Subsection~\ref{Cyclic subtractive algorithms}, we briefly discuss cyclic subtractive algorithms in $\mathbb{R}^3$ and relate them to their corresponding semigroups.

\section{Markov minima and Markov spectrum}
\label{Markov minima and Markov spectrum}

In this section we briefly recall the classical background on Markov minima and
Markov-Davenport forms.

\subsection{Markov minima of indefinite quadratic forms}
\label{Markov minima of decomposable quadratic forms}

The classical theory begins with indefinite binary quadratic forms with real coefficients.
\begin{definition}
Let $f(x, y) = ax^2 + bxy + cy^2$ be a binary quadratic form with discriminant $\Delta(f) = b^2 -4ac>0$. Equivalently, \(f\) is nondegenerate and indefinite. The \textit{Markov minimum} of \(f\) is defined by
\[
m(f)=\inf_{(x,y)\in\mathbb Z^2\setminus\{(0,0)\}} |f(x,y)|.
\]
For forms with $m(f)>0$, the scale-invariant quantity
\[
\mu(f)
    =
    \frac{\sqrt{\Delta(f)}}{m(f)},
\]
is called the
\textit{Markov value} of \(f\). The \textit{Markov spectrum}
is the set of all values \(\mu(f)\), as \(f\) ranges over binary quadratic forms with \(\Delta(f)>0\) and $m(f)>0$.
\end{definition}

For more details on the Markov spectrum we refer to~\cite{Cusick1989}.

\subsection{Markov-Davenport forms}
\label{Markov-Davenport forms}

Markov minima and Markov values can also be studied through matrices.
Throughout this paper, vectors are regarded as column vectors.

\begin{definition}\label{def: markov davenport 2d}
Let $A \in \Mat(2,\r)$. The \textit{Markov-Davenport form} of $A$ is
\[
    f_A(x,y)=\det(v,A v),
    \qquad
    v=(x,y)^T.
\]
\end{definition}

Explicitly, if $A=
    \begin{psmallmatrix}
    a & b\\
    c & d
    \end{psmallmatrix}$, then
\[
    f_A(x,y)
    =
    \det
    \begin{pmatrix}
    x & ax+by\\
    y & cx+dy
    \end{pmatrix}
    =
    cx^2+(d-a)xy-by^2.
\]

\begin{remark}
The discriminant of \(f_A\) is
\[
    \Delta(f_A)
    =(d-a)^2+4bc
    =(\operatorname{tr}A)^2-4\det A
    =\operatorname{Disc}(\chi_A).
\]
Consequently, \(f_A\) is nondegenerate and indefinite if and only if
\(A\) has two distinct real eigenvalues.

Under this hypothesis, the equation \(f_A=0\) defines the union of the
two eigenlines of \(A\). Equivalently, \(f_A\) is a product of two
linear forms, one vanishing on each eigenline.
\end{remark}

\begin{remark}
Conversely, suppose that
\(f=C\ell_1\ell_2\), where \(C\neq0\) and the linear forms
\(\ell_1,\ell_2\) are nonproportional. If \(A\) has the two lines
\(\ker\ell_1\) and \(\ker\ell_2\) as eigenspaces, with distinct
eigenvalues, then \(f_A\) is a nonzero scalar multiple of \(f\).
\end{remark}

The same construction has a natural multidimensional analogue. 

\begin{definition}\label{def: markov davenport nd}
Let $A \in \Mat(k,\r)$. The \textit{Markov--Davenport form} of
$A$ is
\[
    f_A(x_1,\ldots,x_k)
    =
    \det(v,Av,\ldots,A^{k-1}v),
    \qquad
    v=(x_1,\ldots,x_k)^T.
\]
\end{definition}
\begin{remark}
No spectral assumption on \(A\) is imposed in this definition. The form
\(f_A\) is a homogeneous polynomial of degree \(k\), and
\(f_A\not\equiv0\) if and only if there exists \(v\in\mathbb R^k\) such
that $v,Av,\ldots,A^{k-1}v$ are linearly independent.

If \(A\) has \(k\) distinct real eigenvalues
\(\lambda_1,\ldots,\lambda_k\), let
\(S=(u_1\,\cdots\,u_k)\) be a matrix of corresponding eigenvectors and
write
\[
    S^{-1}v=(\ell_1(v),\ldots,\ell_k(v))^T.
\]
Then
\[
    f_A(v)
    =
    \det(S)
    \prod_{1\leq i<j\leq k}(\lambda_j-\lambda_i)
    \prod_{i=1}^k\ell_i(v).
\]
Thus, under this hypothesis, \(f_A\) is a product of \(k\) independent
real linear forms.
\end{remark}

The three-dimensional minimum problem for products of real linear
forms was studied by H.~Davenport
\cite{Davenport1941,Davenport1943,Davenport1947}.

\subsection{Algebraic and geometric Markov numbers}
\label{Algebraic and geometric Markov numbers}

We now introduce two numerical quantities attached to a matrix.

\begin{definition}\label{alg-geom}
Let $M\in \Mat(n,\r)$, and let $f_M$ be its
Markov--Davenport form.
\begin{itemize}
    \item The \textit{geometric Markov number} of $M$, denoted by $m(M)$, is
    \[
        m(M)= \inf_{p \in \mathbb{Z}^n \setminus \{0\}} |f_M(p)|.
    \]
    \item The \textit{algebraic Markov number} of $M$, denoted by $\hat{m}(M)$, is
    \[
        \hat{m}(M) = |f_M(0, \ldots, 0, 1)|.
    \]
\end{itemize}
\end{definition}

By definition,
\[
    m(M)\leq \hat m(M),
\]
but equality may not always hold.
%For Markov reduced $2\times2$ matrices, they agree by definition.
\begin{remark}
If \(M\in\Mat(n,\z)\), then \(f_M\) is integer-valued on
\(\mathbb Z^n\). Consequently, the infimum defining \(m(M)\) is
attained, so it may be replaced by a minimum, and $ m(M)\in\mathbb Z_{\geq0}.$

Moreover, $m(M)>0$ if and only if $\chi_M$ is irreducible over $\mathbb Q$. Indeed, for \(p\in\mathbb Z^n\setminus\{0\}\), one has \(f_M(p)=0\)
if and only if $\operatorname{span}_{\mathbb Q}
    \{p,Mp,\ldots,M^{n-1}p\}$ is a proper \(M\)-invariant rational subspace. Such a subspace exists
if and only if \(\chi_M\) is reducible over \(\mathbb Q\).
\end{remark}
\begin{remark}
Suppose that \(M\) has \(n\) distinct real eigenvalues. If $f_M=L_1\cdots L_n$ is its factorization into real linear forms and \(B\) is the matrix
whose \(i\)-th row is the coefficient vector of \(L_i\), then the
\emph{decomposable-form discriminant} is $ D(f_M)=(\det B)^2.$ This definition is independent of the ordering and normalization of
the factors. For the Markov--Davenport form one has
\[
    D(f_M)
    =
    \operatorname{Disc}(\chi_M)
    =
    \prod_{1\leq i<j\leq n}
    (\lambda_i-\lambda_j)^2.
\]
The corresponding scale-invariant homogeneous minimum is therefore $\frac{m(M)}{\sqrt{D(f_M)}}.$ When \(m(M)>0\), its reciprocal is the associated Markov value. This is the normalization appearing
in the classical multidimensional theory of Davenport; see e.g. \cite{Davenport1941,Davenport1943,Davenport1947}.
\end{remark}

\begin{remark}
The algebraic Markov number is an explicit algebraic expression in the
entries of $M$. In the case $n=2$, if $M=
    \begin{psmallmatrix}
    a & b\\
    c & d
    \end{psmallmatrix}$
then
    $\hat m(M)=|f_M(0,1)|=|b|$.
%Thus, for reduced matrices with $b>0$, the algebraic Markov number is the
%upper-right entry.
\end{remark}

\begin{remark}\label{Non-sym-rem}
The geometric Markov number is invariant under integer change of basis.
More precisely, if $U\in\GL(n,\z)$, then $f_{U^{-1}MU}(v)
    =
    \det(U)^{-1}f_M(Uv).$ Since \(\det(U)=\pm1\) and \(U\) permutes the nonzero points of
\(\mathbb Z^n\), it follows that $m(U^{-1}MU)=m(M)$. 

In contrast, the algebraic Markov number is generally
basis-dependent, since $ \hat m(U^{-1}MU)
    =
    |f_M(Ue_n)|.$ For example, let $M=
    \begin{psmallmatrix}
        0&1\\
        1&3
    \end{psmallmatrix}$ and $U=
    \begin{psmallmatrix}
        1&1\\
        0&1
    \end{psmallmatrix}$. Then $U^{-1}MU
    =
    \begin{psmallmatrix}
        -1&-3\\
        1&4
    \end{psmallmatrix}$
and hence $\hat m(M)=1$ and $\hat m(U^{-1}MU)=3.$
\end{remark}

\begin{remark}\label{Non-sym-rem}
In dimension two, the geometric Markov number is also invariant under
transposition. Indeed, if $R=
    \begin{psmallmatrix}
        0&1\\
        -1&0
    \end{psmallmatrix},$ then $f_{A^T}(v)=-f_A(Rv),$
or equivalently, $f_{A^T}(x,y)=-f_A(y,-x).$
Since \(R\in\GL(2,\z)\), it follows that $ m(A^T)=m(A).$

The algebraic Markov number is generally not invariant under
transposition: if $A=
    \begin{psmallmatrix}
        a&b\\
        c&d
    \end{psmallmatrix},$
then $\hat m(A)=|b|$ and $ \hat m(A^T)=|c|.$
\end{remark}

\section{Basic properties of discrete Markov spectrum}\label{Basic properties of discrete Markov spectrum}
Let us continue with the classical correspondences between Markov numbers,
Farey fractions, Cohn matrices, and perfect matchings of domino snake graphs.
We use these basic facts for the
wug-snake graph theory developed in the following sections.

\subsection{Discrete Markov spectrum and Markov equation}\label{subsection: Discrete Markov spectrum and Markov equation}
The Markov spectrum has a rich structure. In particular, it is discrete below 3. Its smallest three elements are as follows:
\begin{align*}
\sqrt{5} &= 1 + [0; (1)] + [0; (1)]; \\
\sqrt{8} &= 2 + [0; (2)] + [0; (2)]; \\
\frac{\sqrt{221}}{5} &= 2 + [0; (2:1:1:2)] + [0; (1:1:2:2)].
\end{align*}
Here $[0;(a_1:\cdots:a_r)]
    :=
    [0;\overline{a_1,\ldots,a_r}]$ denotes the purely periodic continued fraction with period
\((a_1,\ldots,a_r)\).
The Markov spectrum below 3 was described by A.~Markov in the following theorem from 1879.

\begin{theorem}[A. Markov {\cite{Markoff1880}}]\label{thm: markov 1879}
\begin{enumerate}[label=\textup{(\roman*)}]
    \item The Markov spectrum below 3 consists of the numbers $\frac{\sqrt{9m^2 - 4}}{m}$, where $m$ is a positive integer such that
    \[
    m^2 + m_1^2 + m_2^2 = 3mm_1m_2, \qquad m_2 \le m_1 \le m,
    \]
    for some positive integers $m_1$ and $m_2$.
    \item Let \((m,m_1,m_2)\) satisfy the conditions in\/ \textup{(i)}. If \(m=1\), set \(u=0\). If \(m>1\), let \(u\) be the unique integer
satisfying
    \[
    0<u\leq \frac{m}{2},
    \qquad
    m_2u\equiv\pm m_1\pmod m.
\]
Define \(v\in\mathbb Z\) by
\[
    u^2+1=vm.
\]
Then the quadratic form
\[
    f_{m;m_1,m_2}(x,y)
    =
    mx^2+(3m-2u)xy+(v-3u)y^2
\]
satisfies $\mu(f_{m;m_1,m_2})=\frac{\sqrt{9m^2 - 4}}{m}$, confirming that $\frac{\sqrt{9m^2 - 4}}{m}$ belongs to the Markov spectrum. 
\end{enumerate}

\end{theorem}

\begin{remark}
Note that any form representing an element of the Markov spectrum below 3 is integer congruent to $f_{m;m_1,m_2}$ for some solution of Markov equation $(m,m_1,m_2)$.
\end{remark}
 
\begin{definition}
The Diophantine equation 
$$%\label{eq: markov eq}
  x_1^2+x_2^2 + x_3^2 = 3 x_1x_2x_3,
$$
is called \textit{Markov equation}. 
Its positive integer solutions
$(x_1,x_2,x_3)$ are called \textit{Markov triples}. 
The integers that occur
in such triples are called \textit{Markov numbers}. 
We denote the set of
Markov numbers by $\mathcal{M}$.
\end{definition}

The first Markov triples, considered up to permutation, are
\[
(1,1,1),\quad (1,1,2),\quad (1,2,5),\quad
(1,5,13),\quad (2,5,29),\quad (1,13,34),
\]
and hence
\[
\mathcal M=\{1,2,5,13,29,34,\ldots\}.
\]

\begin{remark}
Throughout this paper, the term \emph{Markov number} refers to a positive
integer occurring in a Markov triple, whereas \(m(f)\) denotes the
Markov minimum of a quadratic form~\(f\). For a Markov form \(f_{m;m_1,m_2}\)  from Markov's theorem
associated with a Markov triple \(m,m_1,m_2\), these quantities coincide
numerically, $ m(f_{m;m_1,m_2})=m.$
\end{remark}

\subsection{Markov tree and Farey indexing}
\label{Markov tree and Farey indexing}

Note that the set of Markov triples is generated by the classical Vieta
involution. Namely, if $(m_1,m_2,m_3)$ is a Markov triple, then replacing
$m_1$ by $m_1'=3m_2m_3-m_1$ again gives a Markov triple. This operation generates the Markov tree.

Note that a Markov triple is called \textit{nonsingular} if it does not have repeating numbers.
There are only two singular Markov triples: $(1,1,1)$ and $(1,1,2)$.

\begin{definition}[Markov tree]
  The \textit{Markov tree} $T_M$ is an infinite rooted binary tree. Its nodes are labeled with nonsingular Markov triples. The root is $(1, 5, 2)$. Every node $(\ell, m, r)$ branches to a left child and a right child:
  \[
      \text{Left child: } (\ell, 3\ell m - r, m) \qquad \text{Right child: } (m, 3mr - \ell, r).
  \]
\end{definition}
\begin{remark}
By \cite[Theorem 3.3]{Aigner2013}, every nonsingular Markov triple occurs
exactly once in $T_M$.
\end{remark}

The Markov tree is naturally indexed by the Farey tree. Recall that the
{\it Farey sum}, or {\it mediant}, of two rational numbers is
\[
    \frac ab\oplus\frac{a'}{b'}
    =
    \frac{a+a'}{b+b'}.
\]

Let $\q_{0,1}$ denote the set of rational numbers in the interval $[0,1]$.
\begin{definition}[Farey tree]
The \textit{Farey tree} $T_F$ is a binary tree labeled with Farey triples of rational numbers in $\q_{0,1}$. The root is $(\frac{0}{1}, \frac{1}{2}, \frac{1}{1})$. A parent node $(\frac{a}{b}, \frac{c}{d}, \frac{a'}{b'})$ branches to a left child and a right child using the mediant rule:
  \[
      \text{Left child: } \left(\frac{a}{b}, \frac{a+c}{b+d}, \frac{c}{d}\right) \qquad \text{Right child: } \left(\frac{c}{d}, \frac{c+a'}{d+b'}, \frac{a'}{b'}\right).
  \]
  \end{definition}
  \begin{remark}
  Every rational number $t \in \q_{0,1}$ appears exactly once as a central element in $T_F$, see \cite[Theorem 3.9]{Aigner2013}. This structure provides a well-defined indexing of the Markov numbers.
  \end{remark}

Comparing the Farey tree with the
Markov tree gives the classical {\it Frobenius map}
\[
    \phi:\mathbb Q_{0,1}\to \mathcal M,\qquad t\mapsto m_t.
\]
For example,
\[
    m_{0/1}=1,\qquad m_{1/1}=2,\qquad m_{1/2}=5.
\]
The rational number $t$ is called the \textit{Farey index} of $m_t$.

\subsection{Cohn matrices}
\label{subsec:cohn-matrices}
H.~Cohn gave a matrix realization of the Farey indexing of Markov numbers.
The starting point is the Fricke trace identity
$$
%  \begin{equation}\label{eq: fricke}
    \tr(A)^2 + \tr(B)^2 + \tr(AB)^2 = \tr(A)\tr(B)\tr(AB) + \tr(ABA^{-1}B^{-1}) + 2
%  \end{equation}
$$
valid for $A,B\in\SL(2,\mathbb Z)$. When the commutator trace is $-2$, this identity reproduces the Markov equation after dividing traces by $3$.
We only need the following consequence of Cohn's
construction.

\begin{theorem}[Cohn matrices; see {\cite{Cohn1955,Aigner2013}}]
There exists a family of matrices
\[
    C_t=
    \begin{pmatrix}
        a_t & m_t\\
        c_t & 3m_t-a_t
    \end{pmatrix}
    \in \SL(2,\mathbb Z),
    \qquad t\in\mathbb Q_{0,1},
\]
such that $m_t$ is the Markov number with Farey index $t$. Moreover, if
$t=r\oplus s$, then the corresponding matrices can be chosen so that $C_t=C_rC_s.$
\end{theorem}

A convenient choice of initial matrices is
\[
    C_{0/1}=
    \begin{pmatrix}
    1&1\\
    1&2
    \end{pmatrix},
    \qquad
    C_{1/1}=
    \begin{pmatrix}
    3&2\\
    4&3
    \end{pmatrix}.
\]

To sum up, in the classical case, the Markov number $m_t$ appears as the
upper-right entry of a matrix product. This multiplicative viewpoint is one
of the main motivations for the semigroup constructions considered later.

\vspace{2mm}

\subsection{Domino snake graphs and their perfect matchings}
\label{subsec:domino-snake-perfect-matchings}

We now recall the combinatorial counterpart of the preceding matrix
construction.

\begin{definition}\label{defi perfect matching}
Let $G=(V,E)$ be a finite graph. A subset $E'\subset E$ is called a {\it perfect matching} if every vertex $v\in V$ is incident to exactly one
edge in $E'$. The number of perfect matchings of $G$ is denoted by $\mu(G)$.
\end{definition}

\begin{example}
Consider the graph
\[
G=
\begin{array}{c}
\includegraphics[height=1cm]{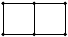}
\end{array}.
\]
Its perfect matchings are precisely
\[
\begin{array}{c}
\includegraphics[height=1cm]{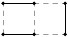}
\end{array}
\qquad
\begin{array}{c}
\includegraphics[height=1cm]{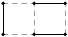}
\end{array}
\qquad
\begin{array}{c}
\includegraphics[height=1cm]{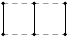}
\end{array},
\]
and hence $\mu(G)=3$.
\end{example}

\begin{definition}
Let $p$ and $q$ be relatively prime positive integers with $p\leq q$. Denote
by $I_{p,q}$ the line segment with endpoints $(0,0)$ and $(q,p)$. The
\textit{snake graph} associated with $I_{p,q}$ is the graph formed by the
vertices and edges of all unit squares in the square lattice that contain
interior points of $I_{p,q}$. We denote this graph by $S(p/q)$.
\end{definition}

The domino snake graph $D(p/q)$ is a bipartite graph constructed from the
snake graph $S(p/q)$ as follows.

\begin{definition}
Let $p$ and $q$ be as above.
\begin{itemize}
    \item Let $V_-$ be the set of lower-right vertices of the unit squares
    in $S(p/q)$.

    \item Subdivide each unit square in $S(p/q)$ into four congruent
    subsquares, which we call \textit{quadrants}.

    \item Select all quadrants that contain a vertex of $V_-$.
\end{itemize}
The graph formed by the vertices and edges of the selected quadrants is
called the \textit{domino snake graph} associated with $p/q$ and is denoted
by $D(p/q)$.
\end{definition}

For further details, see~\cite{Aigner2013}.

\begin{remark}
Although the terminology suggests a domino tiling picture, the number of
quadrants in a domino snake graph is odd.
\end{remark}

\begin{example}
Consider the segment $I_{2,3}$, shown as the bold line in the figure below:
\[
\begin{array}{c}
\includegraphics[height=3cm]{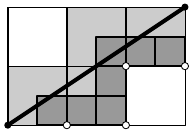}
\end{array}
\]
The snake graph $S(2/3)$ is formed by the vertices and edges of the
light-grey unit squares. In this example, the set $V_-$ consists of four
points, one for each light-grey square; these points are shown in white.
There are seven quadrants of the snake graph containing points of $V_-$.
They are shown in dark grey and form the domino snake graph $D(2/3)$.
\end{example}

The following theorem is the classical bridge between Markov numbers and
perfect matchings.

\begin{theorem}[Aigner {\cite[Theorem 7.12]{Aigner2013}}]
\label{thm:aigner-domino-markov}
For every $t\in\mathbb Q_{0,1}$,
\[
    \mu(D(t))=m_t,
\]
where $m_t$ is the Markov number with Farey index $t$.
\end{theorem}

\begin{remark}[Transfer-matrix viewpoint]
The proof of Theorem~\ref{thm:aigner-domino-markov} can be interpreted using
local transfer matrices. As one moves along the domino snake graph, the
numbers of partial perfect matchings form a two-dimensional state vector.
The elementary local moves act on this vector by the matrices
\[
    \begin{pmatrix}
    1&1\\
    1&2
    \end{pmatrix}
    \qquad\text{and}\qquad
    \begin{pmatrix}
    3&2\\
    4&3
    \end{pmatrix}.
\]
Thus the total number of perfect matchings is recovered as the upper-right
entry of the corresponding Cohn matrix. This transfer-matrix mechanism is
precisely what will be generalized below by wug-snake graphs and their
continuant matrices.

The classical picture can therefore be summarized as follows. Markov numbers
are simultaneously:
\begin{itemize}
    \item entries of Cohn matrices indexed by the Farey tree;
    \item perfect matching numbers of domino snake graphs;
    \item values arising from transfer-matrix recurrences.
\end{itemize}
In the remainder of the paper we extend this triangle of ideas. Cohn matrices
are replaced by semigroups of integer matrices, domino snake graphs are
replaced by wug-snake graphs, and the associated transfer recurrences are
encoded by continuant matrices.
\end{remark}

\section{Wug-snake graphs}
\label{Generalised universal weighted snake graph}

In this section we introduce wug-snake graphs and develop their basic
linear-algebraic properties. These graphs are designed so that their weighted
perfect matching numbers satisfy prescribed linear recurrence sequences. We then use
them to construct polyomino wug-tiles for matrices and relate the resulting
determinants to Markov-Davenport forms and to algebraic and geometric Markov
numbers.

The section is organised as follows. In
Subsection~\ref{Definition of wug-snake graphs and their basic properties}
we define wug-snake graphs and state their fundamental recurrence formula;
the proof is given in Subsection~\ref{Continuant and proof}. In
Subsection~\ref{Heads and bodies of wug-snake graphs, their determinants}
we introduce heads, bodies, state vectors, and wug-snake determinants. In
Subsection~\ref{Polyomino wug-tiles for matrices} we define polyomino
wug-tiles. Subsection~\ref{subsection: Canonical construction} gives
canonical constructions of heads and matrix tiles. In
Subsection~\ref{Polyomino wug-tiles and Markov-Davenport forms} we prove
the determinant formula for Markov-Davenport forms. 
Further Subsections~\ref{Linear recursions and lower companion matrices} and~\ref{Recurrence relations for sequences and corresponding wug-graphs}
relate the construction to linear recurrences and lower companion matrices.
In Subsection~\ref{Lower companion matrices decomposition} we prove that every integer matrix can be represented as a product of lower companion matrices; in this context we briefly recall the reduced matrices in Subsection~\ref{Semigroup of reduced matrices}.
We say a few words on the rank of wug-tiles representing matrices in Subsection~\ref{A few words on the rank of wug-tiles representing matrices}.
 Finally, Subsection~\ref{Wug-snake graphs for CW-complexes} briefly discusses
face matchings in CW-complexes.

\subsection{Wug-snake graphs and their basic recurrence}
\label{Definition of wug-snake graphs and their basic properties}

\subsubsection{Definition and examples} We first define the ambient class of ordered weighted bipartite graphs.

\begin{definition}
Let $n\in\z_{>0}\cup\{\infty\}$. An \textit{ordered weighted bipartite
graph} $K_n(w)$ consists of two ordered vertex sets
\[
    U=\{u_1,\ldots,u_n\},
    \qquad
    V=\{v_1,\ldots,v_n\},
\]
and a weight matrix $w=(w_{i,j})$, where $w_{i,j} \in \z$ is the weight of the
edge $u_i v_j$. If $w_{i,j}=0$, the edge $u_i v_j$ is absent.
\end{definition}

We now impose the structural condition that defines wug-snake graphs.

\begin{definition}\label{def: super upper triangular}
A matrix $w = (w_{i,j})$ is called \textit{super upper triangular} if 
\[
    w_{i,j}=0
    \qquad
    \text{for all } i>j+1.
\]
Equivalently, all entries below the subdiagonal are zero.
\end{definition}

This specific configuration leads directly to our central concept:

\begin{definition}[Wug-snake graph]\label{def:continuant-wug}
An ordered weighted bipartite graph $K_n(w)$ is called a
\textit{weighted universal generalised snake graph}, or \textit{wug-snake
graph}, if its weight matrix $w$ is super upper triangular and
\[
    w_{i+1,i}=1
\]
for every admissible $i$. In other words, all subdiagonal edges have
weight $1$.
\end{definition}
\begin{definition}[Weighted perfect matching number]
Consider a graph $K_n(w)$, and let $P$ be the set of all its perfect matchings.
The \textit{weighted perfect matching number} of $K_n(w)$, denoted by
$\mu(K_n(w))$, is
\[
    \mu(K_n(w))
    =
    \sum_{P}
    \prod_{u_i v_j\in P} w_{i,j}.
\]
When all nonzero weights are $1$, this is the ordinary
number of perfect matchings.
\end{definition}
\begin{remark}
When weights are not all nonnegative, $\mu(K_n(w))$ should be interpreted
as a weighted matching sum rather than as an ordinary counting function.
\end{remark}
This framework naturally generalizes traditional models, as stated in the following proposition.
\begin{proposition}\label{prop: domino are wug-snake}
Every domino snake graph can be represented as a wug-snake graph.
\end{proposition}
\begin{proof}
A domino snake graph is a bipartite subgraph of the square lattice. Choose a
bipartition $U\sqcup V$ and order the vertices by following the domino snake
from one end to the other. With this ordering, the boundary edges of the
domino snake give edges $u_{i+1}v_i$ for $i=1,\ldots,n-1$, all of weight
$1$. Since the graph is obtained by successively attaching domino blocks
along edges, no edge can connect $u_i$ to $v_j$ with $i>j+1$. Thus the
corresponding weight matrix is super upper triangular and has all
subdiagonal entries equal to $1$. Hence the domino snake graph is
represented by a wug-snake graph.
\end{proof}

We illustrate this proposition with a simple example.

\begin{example}\label{ex:perm-det}
Consider the domino snake graph $D\left(2/3\right)$:
\[
    \vcenter{\hbox{\includegraphics[width=3cm]{figures/snake29-1.pdf}}}
\]

We can redraw $D\left(2/3\right)$ as a wug-snake graph $K_8(w)$ by assigning every edge a weight of $1$:
\[
    \vcenter{\hbox{\includegraphics[width=3cm]{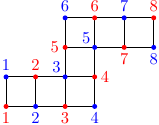}}}
    \qquad
    \vcenter{\hbox{\includegraphics[width=5cm]{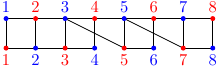}}}
    \qquad
    \vcenter{\hbox{\scalebox{1.05}{$w =\left(\begin{smallmatrix}
        1 &  1 &  0 &  0 &  0 &  0 &  0 & 0 \\
       1 &  1 &  1 &  0 &  0 &  0 &  0 & 0 \\
        0 & 1 &  1 &  1 &  1 &  0 &  0 & 0 \\
        0 &  0 & 1 &  1 &  0 &  0 &  0 & 0 \\
        0 &  0 &  0 & 1 &  1 &  1 &  1 & 0 \\
        0 &  0 &  0 &  0 & 1 &  1 &  0 & 0 \\
        0 &  0 &  0 &  0 &  0 & 1 &  1 & 1 \\
        0 &  0 &  0 &  0 &  0 &  0 & 1 & 1
    \end{smallmatrix}\right)$}}}
\]

The figures above illustrate the enumeration of the bipartite vertex sets $U$ (labelled in red) and $V$ (labelled in blue) for both the domino snake graph (left) and its corresponding wug-snake representation (middle) and weight matrix (right).
\end{example}

We continue by formally defining the subgraphs and associated sequences of wug-snake graphs.

\begin{definition}[Perfect matching sequence]
Let $K_n(w)$ be a wug-snake graph. For $0\leq i\leq n$, we denote by
$K_i(w)$ the induced wug-snake subgraph on the vertices
\[
    u_1,\ldots,u_i,
    \qquad
    v_1,\ldots,v_i.
\]
We set $K_0(w)=\varnothing$ and use the convention $\mu(K_0(w))=1.$
The sequence $(K_i(w))_{i=0}^n$ is called the \textit{filtration} of
$K_n(w)$, and the sequence
\[
    \bigl(\mu(K_i(w))\bigr)_{i=0}^n
\]
is called its \textit{perfect matching sequence}. When listing a perfect
matching sequence explicitly, we usually omit the conventional initial
term \(\mu(K_0(w))=1\) and write $ \bigl(\mu(K_1(w)),\ldots,\mu(K_n(w))\bigr).$
\end{definition}

\begin{definition}[Continuant matrix]\label{def:continuant matrix}
Let $K_n(w)$ be a wug-snake graph. Its \textit{continuant matrix} $C_n(w)$ is the $n\times n$ matrix obtained from the upper-left $n\times n$ part of $w$ by replacing each subdiagonal entry $1$ by $-1$. Thus
\[
    (C_n(w))_{i,j}
    =
    \begin{cases}
    -1, & i=j+1,\\
    w_{i,j}, & \text{otherwise}.
    \end{cases}
\]
\end{definition}

The following proposition establishes a recursive formula relating the wug-snake graph $K_n(w)$ to the subgraphs within its filtration. We discuss continuants in more details later in Example~\ref{ex:continuant-matrix}.

\begin{theorem}\label{theorem:wug-det}
Let $K_n(w)$ be a wug-snake graph. Then
\[
    \mu(K_n(w))=\det(C_n(w)).
\]
Moreover, the perfect matching sequence satisfies the recurrence
\[
    \mu(K_n(w))
    =
    \sum_{i=1}^n w_{i,n}\,\mu(K_{i-1}(w)).
\]
\end{theorem}

We defer the proof of this theorem a bit later in Subsection~\ref{Continuant and proof}.

\begin{remark}[Positive integer weights]
A positive integer weight $a$ on an edge may be interpreted as an edge of
multiplicity $a$. Equivalently, one may replace such a weighted edge by an
unweighted local configuration:
$$
\begin{array}{c}
\includegraphics[width=5cm]{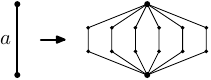}
\end{array}
$$
More precisely, depending on whether the multiple edge is used in the
matching, one replaces it by one of the two configurations
$$
\begin{array}{c}
\includegraphics[width=5cm]{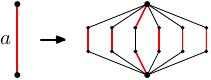}
\end{array}
\qquad \quad
\begin{array}{c}
\includegraphics[width=5cm]{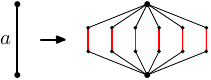}
\end{array}
$$
Thus positive integer weights can be eliminated at the cost of enlarging
the graph.
\end{remark}

\subsection{Continuant matrices for wug-snake graphs}
\label{Continuant and proof}

We now prove Theorem \ref{theorem:wug-det} by relating the continuant matrix to the perfect matching number.

\subsubsection{Continuant matrices}
\label{Continuants of wug-snake graphs}

Let $K_n(w)$ be a wug-snake graph with weight matrix $w=(w_{i,j})$. Recall
that $w$ is super upper triangular and that all subdiagonal entries satisfy
$w_{i+1,i}=1$. Recall we have defined the continuant matrices $C_n(w)$ associated with $K_n(w)$ in Definition \ref{def:continuant matrix}.

\begin{remark}
The matrix $C_n(w)$ is analogous to the continuant matrices appearing in
the theory of continued fractions. In the present setting, it plays the role
of a determinant model for perfect matchings of wug-snake graphs.
\end{remark}

\begin{example}\label{ex:continuant-matrix}
Consider the wug-snake graph $K_8(w)$ from Example~\ref{ex:perm-det}, where
all nonzero edge weights are equal to $1$:
\[
    \vcenter{\hbox{\includegraphics[width=5cm]{figures/snake29-3.pdf}}}
\]
The corresponding continuant matrix is
\[
    C_8(w)=
    \left(\begin{smallmatrix}
        1 &  1 &  0 &  0 &  0 &  0 &  0 & 0 \\
       -1 &  1 &  1 &  0 &  0 &  0 &  0 & 0 \\
        0 & -1 &  1 &  1 &  1 &  0 &  0 & 0 \\
        0 &  0 & -1 &  1 &  0 &  0 &  0 & 0 \\
        0 &  0 &  0 & -1 &  1 &  1 &  1 & 0 \\
        0 &  0 &  0 &  0 & -1 &  1 &  0 & 0 \\
        0 &  0 &  0 &  0 &  0 & -1 &  1 & 1 \\
        0 &  0 &  0 &  0 &  0 &  0 & -1 & 1
    \end{smallmatrix}\right).
\]
A direct computation gives $\det(C_8(w))=29.$ This agrees with the number of perfect matchings of $K_8(w)$.
\end{example}

\begin{remark}
In Example~\ref{ex:continuant-matrix}, the entries $-1$ on the subdiagonal
correspond to the green edges in the following graph:
\[
    \vcenter{\hbox{\includegraphics[width=4.8cm]{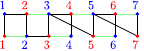}}}
\]
These edges are forced whenever one matches a later vertex $v_k$ to an
earlier vertex $u_i$.
\end{remark}

\subsubsection{Perfect matching numbers and continuants}
\label{Matching numbers and continuants}

We now prove that the determinant of the continuant matrix computes the
weighted perfect matching number of the corresponding wug-snake graph.

\begin{proposition}\label{kn-cn}
Let $K_n(w)$ be a wug-snake graph, and let $C_n(w)$ be its continuant
matrix. Then
\[
    \mu(K_n(w))=\det(C_n(w)).
\]
Moreover, the sequence $\mu(K_k(w))$ satisfies the recurrence
\[
    \mu(K_k(w))
    =
    \sum_{i=1}^k w_{i,k}\mu(K_{i-1}(w)),
    \qquad
    k=1,\ldots,n,
\]
with the convention $\mu(K_0(w))=1$.
\end{proposition}

\begin{proof}
For $1\leq k\leq n$, let $C_k(w)$ be the $k\times k$ upper-left principal
submatrix of $C_n(w)$. We prove by induction on $k$ that $\mu(K_k(w))=\det(C_k(w))$.

For $k=0$, both sides are equal to $1$ by convention. For $k=1$, the graph
$K_1(w)$ consists of two vertices joined by an edge of weight $w_{1,1}$.
Thus $\mu(K_1(w))=w_{1,1}=\det(C_1(w))$.

Assume now that the statement holds for all smaller indices. We first note
that $\det(C_k(w))$ satisfies the recurrence
\[
    \det(C_k(w))
    =
    \sum_{i=1}^k w_{i,k}\det(C_{i-1}(w)).
\]
Indeed, this follows by expanding $\det(C_k(w))$ along the last column. The
entry in row $i$ of the last column is $w_{i,k}$. After deleting row $i$ and
the last column, the lower-right part of the remaining matrix is forced by
the subdiagonal entries $-1$. The corresponding cofactor sign is exactly
cancelled by the product of these $-1$ entries, leaving the contribution $ w_{i,k}\det(C_{i-1}(w)).$

It remains to show that $\mu(K_k(w))$ satisfies the same recurrence. In any
perfect matching of $K_k(w)$, the vertex $v_k$ must be matched to exactly
one of the vertices $u_i$, where $1\leq i\leq k$. Suppose that the edge
$u_i v_k$ is chosen. This contributes the weight $w_{i,k}$. Then all
intermediate vertices are forced to be matched by the subdiagonal edges
\[
    u_kv_{k-1},\quad u_{k-1}v_{k-2},\quad \ldots,\quad u_{i+1}v_i.
\]
Each of these edges has weight $1$. After removing these forced matched
vertices, the remaining graph is precisely $K_{i-1}(w)$. Hence the total
weighted contribution of all perfect matchings containing the edge
$u_i v_k$ is $w_{i,k}\mu(K_{i-1}(w)).$

Summing over all possible choices of $i$ gives
\[
    \mu(K_k(w))
    =
    \sum_{i=1}^k w_{i,k}\mu(K_{i-1}(w)).
\]

Thus $\mu(K_k(w))$ and $\det(C_k(w))$ satisfy the same recurrence with the
same initial condition. By induction,
\[
    \mu(K_k(w))=\det(C_k(w))
\]
for all $k=0,\ldots,n$. In particular, the first statement of the proposition holds.
\end{proof}

\begin{proof}[Proof of Theorem~\ref{theorem:wug-det}]
This is immediate from Proposition~\ref{kn-cn}.
\end{proof}

\subsection{Heads, bodies, and wug-snake determinants}
\label{Heads and bodies of wug-snake graphs, their determinants}

\subsubsection{Decomposition and numerical invariants}

We now decompose a wug-snake graph into two parts: an initial segment
called the \textit{head}, and the remainder
called the \textit{body}. We begin
with the formal definitions and then introduce the associated numerical
invariants.

\begin{definition}\label{def:head-body}
Let $K_n(w)$ be a wug-snake graph, and let $d$ be an integer with
$1\leq d\leq n$. We decompose $K_n(w)$ as follows.
\begin{itemize}
    \item The subgraph $H=K_d(w)$ is called the \textit{$d$-head}, or
    simply the \textit{head}, of $K_n(w)$.

    \item The remaining part $B=K_n(w)\setminus K_d(w)$ is called the
    \textit{$(n-d)$-body}, or simply the \textit{body}, of $K_n(w)$.

    \item The set of edges connecting vertices of $H$ to vertices of $B$
    is called the \textit{neck} of $K_n(w)$. By convention, the neck is
    regarded as part of the body data.
\end{itemize}
We write this decomposition as $K_n(w)=HB$.
\end{definition}

We associate the following numerical invariants to this decomposition.

\begin{definition}\label{def:rank-neck}
Let $K_n(w)=HB$ be as above.
\begin{itemize}
    \item The \textit{lengths} of $H$, $B$, and $K_n(w)$ are $d$,
    $n-d$, and $n$, respectively.

    \item The \textit{rank} of $K_n(w)$ is
    \[
        \max\{j-i+1 : w_{i,j}\neq 0,\; i\leq j\}.
    \]
    Equivalently, it is one more than the index of the highest nonzero
    superdiagonal of $w$.
    
    The rank of a body \(B\), including its neck data, is the maximum width
\(j-i+1\) of a nonzero upper-triangular weight belonging to \(B\).
This quantity is invariant under attaching \(B\) to different heads of
the same length.

    \item The \textit{neck-length} of $B$ is $d-\lambda$, where $\lambda$
    is the number of initial zero rows in the submatrix of $w$ formed by
    columns $d+1,\ldots,n$.
\end{itemize}
\end{definition}

%\begin{remark}
%The neck records how a body is attached to a %head. Thus, when we multiply
%bodies below, the attachment data are part of the structure.
%\end{remark}

Let us illustrate the definition with the following example.

\begin{example}\label{ex:head-body-decomposition}
Consider the following wug-snake graph $K_9(w)$:
$$
\includegraphics[height=2cm]{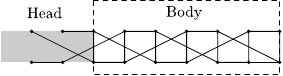}
$$
Note that the first grey square appears to have two vertices missing. This
is because $K_i$ has $2i$ vertices: $K_1$ corresponds to $2$ vertices,
$K_2$ to $4$, $K_3$ to $6$, and so on.

\vspace{2mm}

Its $3$-head and $6$-body are shown separately:
$$
\includegraphics[height=1cm]{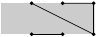}
\qquad\qquad
\includegraphics[height=2cm]{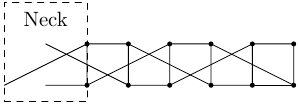}
$$

The corresponding weight matrix $w$ is displayed below:
$$
\includegraphics[height=5cm]{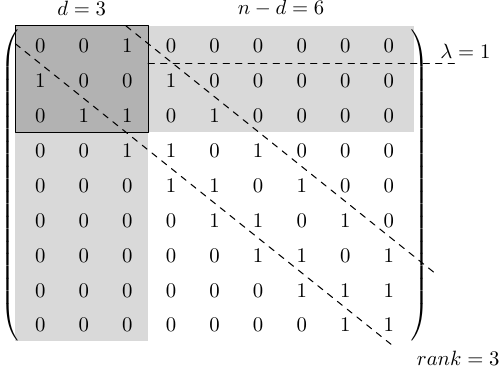}
$$
The dark grey $3\times 3$ block in the upper-left corner is the weight
matrix of the head. The light grey entries indicate the neck: the nonzero
positions in this region are the edges connecting the head to the body.

\vspace{2mm}

\noindent\textit{Rank.}\quad
The last nonzero superdiagonal of $w$ is the $2$-superdiagonal. Hence
the rank of this wug-snake graph is $3$.

\vspace{2mm}

\noindent\textit{Neck.}\quad
The neck consists of three edges: two from the upper-triangular part of
$w$ (whose positions vary from example to example) and one from the
subdiagonal (which is always present, by definition of a wug-snake graph).

\vspace{2mm}

\noindent\textit{Neck-length.}\quad
Removing the first three columns from $w$, we observe that the first row
of the resulting submatrix is zero, while the second row is nonzero. Thus
$\lambda=1$ and the neck-length equals $d-\lambda=3-1=2$. As we shall see
below, a body with neck-length $2$ can be attached to any head of length
at least $2$.
\end{example}

\subsubsection{Product of bodies, and wug-snake determinants}

\begin{definition}[Products of bodies and heads]
\quad
\begin{itemize}
\item
If \(B_1\) and \(B_2\) are two compatible bodies, their \emph{product}
\(B_1B_2\) is obtained by attaching the neck of \(B_2\) to the terminal
vertices of \(B_1\), respecting the orderings of the \(U\)- and
\(V\)-vertices.

\item
We denote by \(B^k\) the product of \(k\) copies of \(B\). We set
\(B^1=B\), and \(B^0\) is the empty body, containing no vertices.

\item
Any head is regarded as a body of neck-length \(0\), and therefore heads
may also be used in such products.
\end{itemize}
\end{definition}

\begin{remark}
With this convention, the product $B_1B_2$ means that $B_2$ is attached
after~$B_1$.
\end{remark}

\begin{example}
Consider the 1-body $B_1$ and the 2-body $B_2$ as follows:
$$
B_1=
\begin{array}{c}
\includegraphics[height=1.4cm]{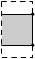}
\end{array}
\qquad \hbox{and}\qquad 
B_2=
\begin{array}{c}
\includegraphics[height=1.4cm]{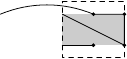}.
\end{array}
$$
Then the 3-body $B_1B_2$ is defined as
$$
B_1B_2=
\begin{array}{c}
\includegraphics[height=1.4cm]{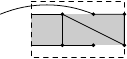}.
\end{array}
$$
\end{example}

\begin{definition}[State vector]\label{def:sum-wugs}
Let $K=HB$ be a wug-snake graph whose head has length $d$. For
$k\geq 0$.  The \textit{$k$-th state vector}
\[
    v_k(HB)\in \z^d
\]
is the vector consisting of the last $d$ entries of the perfect matching
sequence of $HB^k$. In particular, $v_0(HB)$ is the $0$-th state vector of the
head $H$.
%\\
%In case if $B$ is an empty body (i.e. the length of $B$ is zero), we write $v_k(H)$ instead of $v_k(HB)$.
\end{definition}

\begin{definition}[Wug-snake determinant]\label{Markov-MDavenport-det}
Let $K=HB$ be a wug-snake graph whose head has length $d$. The
\textit{wug-snake determinant} of $K$ is
\[
    \det(K)
    =
    \det\bigl(v_0(HB),v_1(HB),\ldots,v_{d-1}(HB)\bigr).
\]
\end{definition}

\begin{remark}
In the classical two-dimensional setting, if the head has state vector
$(1,0)$ or $(0,1)$, then the wug-snake determinant reduces to one of
the last entries of the perfect matching sequence. This recovers the usual
appearance of Markov numbers in classical snake graphs.
\end{remark}

\subsection{Polyomino wug-tiles for matrices}
\label{Polyomino wug-tiles for matrices}

We now explain how bodies of wug-snake graphs represent matrices. This is
the key link between perfect matchings and Markov-Davenport forms.

\subsubsection{Snake operators and polyomino wug-tiles}

\begin{definition}[Snake value]
Let \(B\) be a body whose neck-length does not exceed \(d\). For
\(x\in\mathbb Z^d\), choose a head \(H_x\) of length \(d\) with state
vector \(v_0(H_xB)=x\). We define
\[
    {\rm sv}_B^d(x)=v_1(H_xB)
\]
to be the \emph{snake value} at $x$. 
\end{definition}

Note that the recurrence formula shows that ${\rm sv}_B^d(x)$ depends only on \(x\) and
the weight data of \(B\), and not on the choice of \(H_x\). 

\begin{definition}[Snake operator]\label{def: snake operator}
We say that the operator \(\mathcal S_B^d:\mathbb Z^d\to\mathbb Z^d\) defined by
$$
x\mapsto {\rm sv}_B^d(x)
$$
is the {\it snake operator} for $B$. 
The snake operator is linear; its matrix is
denoted by \(M_B^d\).

\vspace{1mm}

\noindent
For \(k\geq0\), define
\[
    \mathcal S_{B^k}^d=(\mathcal S_B^d)^k.
\]
This map is called the \emph{k-step snake operator of $B$}.
\end{definition}

\begin{definition}[Polyomino wug-tile]\label{def: Polyomino wug-tile}
Let $A\in\Mat(d,\z)$. If there exists a wug-snake graph $K=HB$ whose head has length $d$, such that the 1-step snake operator of $B$ satisfies $M_B^d=A$, then $B$ is called a \textit{polyomino wug-tile} for $A$. Denote by $\mathcal T_d(A)$
the set of all polyomino wug-tiles for $A$.
\end{definition}

We will show in Subsection \ref{subsection: Canonical construction} that $\mathcal T_d(A)$ is nonempty.

\vspace{2mm}

Conversely, for a fixed dimension $d$, every body $B$ with neck-length at
most $d$ determines a unique matrix $M^d_{B}$. Thus $B$ is a polyomino
wug-tile for this unique matrix. The condition on the neck-length ensures
that the $m$-dimensional state vector contains enough information to attach
$B$ and compute the new state.

\begin{example}
Consider the body $B$  (shown in white) of a wug-snake graph:
\[
\begin{array}{c}
\includegraphics[height=3cm]{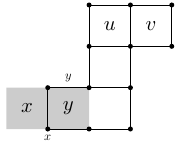}
\end{array}.
\]
All displayed edge weights are equal to $1$ except for the ones in the head marked by $x$ and $y$. Attach this body to a head $H$ whose
state vector is
\[
    v_0(HB)=
    \begin{pmatrix}
    x\\
    y
    \end{pmatrix}.
\]
The perfect matching sequence along the four white squares is
\[
    x+y,\qquad 2x+y,\qquad u=3x+2y,\qquad v=4x+3y.
\]
Hence the state vector after attaching the body is
\[
    v_1(HB)=
    \begin{pmatrix}
    u\\
    v
    \end{pmatrix}
    =
    \begin{pmatrix}
    3 & 2\\
    4 & 3
    \end{pmatrix}
    \begin{pmatrix}
    x\\
    y
    \end{pmatrix}.
\]
Thus the 1-step snake operator of $B$ is
\[
    M^2_{B}=
    \begin{pmatrix}
    3 & 2\\
    4 & 3
    \end{pmatrix},
\]
and therefore $B$ is a polyomino wug-tile for this matrix. Notice that the
rank of the wug-snake graph is $3$, whereas the induced snake operator acts
on a two-dimensional state space.
\end{example}

\begin{proposition}\label{product of tiles}
Let $B_1$ and $B_2$ be polyomino wug-tiles for matrices
$M_1,M_2 \in\Mat(m,\z)$, respectively. Then $B_1B_2$ is a polyomino wug-tile for
$M_2M_1$.
\end{proposition}
\begin{proof}
Attaching $B_1$ sends the state vector $v$ to $M_1v$. Attaching
$B_2$ afterwards sends $M_1v$ to $M_2M_1v$. Hence the body $B_1B_2$
represents $M_2M_1$.
\end{proof}

\subsection{Canonical heads and canonical polyominoes}\label{subsection: Canonical construction}
Given a sequence of integers $(v_1,\ldots,v_d)$ of length $d$ and a matrix $A \in \Mat(m,\z)$, we provide a canonical construction of a wug-snake graph $K = HB$ which encodes both data.
\subsubsection{Canonical realisations of perfect matching sequences}\label{subsubssection: Canonical realisations of perfect matching sequences}
Given a sequence of integers $v = (v_1,\ldots,v_d)$, the set of all heads whose perfect matching sequence is $v$ is denoted by $\Head(v)$. We start with a canonical construction of a head $H_v$ with a prescribed
perfect matching sequence.

\begin{definition}
Let $v = (v_1,\ldots,v_d)$ be a sequence of integers. Its
\textit{canonical realisation} is the wug-snake graph $H_v$ with weight matrix
\[
w_{i,j}
=
\begin{cases}
v_j, & i=1,\\
1,   & i=j+1,\\
0,   & \text{otherwise}.
\end{cases} \qquad \text{ for }1 \leq i,j\leq d
\]
\end{definition}

\begin{example}
The canonical realisation of the sequence $(2,3,4,5)$
 is the wug-snake graph with its weight matrix shown below:
$$
\begin{array}{c}
\includegraphics[height=2cm]{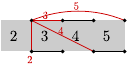}
\end{array} \qquad w = \begin{pmatrix}
 2&3&4&5\\
 1&0&0&0\\
 0&1&0&0\\
 0&0&1&0
\end{pmatrix}.
$$
All black edges have weight 1.
\end{example}

\begin{proposition}\label{canonical-head-prop}
The perfect matching sequence of the canonical realisation of
$(v_1,\ldots,v_d)$ is precisely $(v_1,\ldots,v_d)$. In particular, $H_v \in \Head(v)$.
\end{proposition}

\begin{proof}
By Theorem~\ref{theorem:wug-det}, the sequence satisfies
\[
    \mu(K_j)=\sum_{i=1}^j w_{i,j}\mu(K_{i-1}).
\]
For the canonical realisation, the only nonzero entry in column $j$ among
the rows $1,\ldots,j$ is $w_{1,j}=v_j$. Hence
\[
    \mu(K_j)=v_j\mu(K_0)=v_j.
\]
This proves the claim.
\end{proof}
\subsubsection{Canonical polyomino for matrices}\label{subsubsection: Canonical polyomino for matrices}
For any integer matrix $A \in \Mat(m,\z)$, we now construct a canonical polyomino wug-tile $B_{\mathrm{can}}(A)$.

\begin{definition}
Let $A=(a_{i,j})\in \Mat(m,\z)$. The \textit{canonical polyomino} of $A$, denoted by $B_{\mathrm{can}}(A)$, is the body of length $m$ whose body weight data are encoded by the
following $2m\times m$ matrix:
\[
\left(
\begin{array}{ccccc}
0&0&\ldots&0&0\\[2pt]
a_{1,1}&a_{1,2}&\ldots& a_{1,m-1}&a_{1,m}\\
\vdots&\vdots&\ddots &\vdots &\vdots\\
a_{m,1}&a_{m,2}&\ldots& a_{m,m-1}&a_{m,m}\\[2pt]
1&0&\ldots&0&0\\
0&1&\ldots&0&0\\
\vdots&\vdots&\ddots &\vdots &\vdots\\
0&0&\ldots&1&0
\end{array}
\right).
\]
The first row of zeros is the row inherited from the head: it records the
fact that the last vertex of the head connects to the body only through the
subdiagonal edge (which has weight $1$ by definition). The middle block
contains the entries of $A$, and the lower block provides the subdiagonal
edges that propagate the state vector forward.
\end{definition}

\begin{example}
Let
$
A=
\begin{pmatrix}
a_{11} & a_{12}\\
a_{21} & a_{22}
\end{pmatrix}.
$
Then by definition, the canonical polyomino $B_{\mathrm{can}}(A)$ has weight matrix
\[
\begin{pmatrix}
0&0\\
a_{11}&a_{12}\\
a_{21}&a_{22}\\
1&0
\end{pmatrix}.
\]
Suppose $v=(x,y)$. For the canonical realisation $H_v$, the weight matrix of the full wug-snake graph $H_vB_{\mathrm{can}}(A)$ is obtained by
placing the head's weight matrix in the first $2$ columns and the body's
weight matrix in the remaining columns, with the neck entries occupying
the overlap region. This gives the $4\times 4$ weight matrix
\[
\begin{pmatrix}
x&y&0&0\\
1&0&a_{11}&a_{12}\\
0&1&a_{21}&a_{22}\\
0&0&1&0
\end{pmatrix}.
\]
The corresponding wug-snake graph is shown below:
\[
\begin{array}{c}
\includegraphics[height=2.3cm]{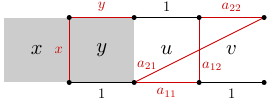}
\end{array}
\]
Here the first two columns represent the canonical realisation $H_v$, while the last two
columns represent the canonical polyomino $B_{\mathrm{can}}(A)$. By construction, attaching
$B_{\mathrm{can}}(A)$ sends the state vector $(x,y)^T$ to
\[
\begin{pmatrix}
u\\
v
\end{pmatrix}
=
\begin{pmatrix}
a_{11} & a_{12}\\
a_{21} & a_{22}
\end{pmatrix}
\begin{pmatrix}
x\\
y
\end{pmatrix}.
\]
Thus the snake operator of the canonical polyomino is precisely $A$.
\end{example}

\begin{proposition}\label{rank-canonic}
Let $A\in\Mat(m,\z)$. Then $B_{\mathrm{can}}(A)\in\mathcal T_m(A)$. In particular, every integer matrix admits a polyomino wug-tile
representation. Moreover, the rank of $B_{\mathrm{can}}(A)$ is at most
$2m-1$.
\end{proposition}

\begin{proof}
The statement follows directly from the construction: the body acts on the
state vector by multiplication with $A$. The displayed construction uses a
recurrence of rank at most $2m-1$. This rank may drop when some of the
initial rows involved in the construction are zero.
\end{proof}

Later in Subsection~\ref{A few words on the rank of wug-tiles representing matrices} we show that any matrix of size $m$ admits  wug-tile representations of rank at most $m$ (see Theorem~\ref{transFrob-tiles}).

\subsection{Polyomino wug-tiles and Markov-Davenport forms}
\label{Polyomino wug-tiles and Markov-Davenport forms}

By Proposition~\ref{rank-canonic} and \ref{canonical-head-prop}, every integer sequence of length $d$ admits at least one realisation, and every matrix $A\in\Mat(d,\z)$ admits at
least one polyomino wug-tile representation. Thus the following theorem applies to  every integer sequence of length $d$ and every integer matrix $A$. Moreover,
the resulting determinant depends only on $A$ and on the state vector of the
head, not on the particular tile chosen; this independence is made explicit
in Corollary~\ref{cor:tile-independence}.

\begin{theorem}[\textbf{Values of the Markov-Davenport form}]\label{MD-det}
Let $A\in\Mat(d,\z)$, and let $B\in\mathcal T_d(A)$ be any polyomino
wug-tile for $A$. 
For a sequence $x=(x_1,\ldots,x_d)$, let $H \in \Head(x)$ be any head whose perfect matching sequence is $x$.
Then
\[
    f_A(x_1,\ldots,x_d)=\det(HB),
\]
where $f_A$ is the Markov-Davenport form of $A$ and $\det(HB)$ is the wug-snake determinant, i.e. the determinant
formed from the state vectors of $H,HB,\ldots,HB^{d-1}$.
\end{theorem}

The proof is immediate once the framework is in place, reflecting the
naturality of the construction.

\begin{proof}
For $x=(x_1,\ldots,x_d)^T$. By construction, the state vector of any $H \in \Head(x)$ is
$x$. Since $B$ is a polyomino wug-tile for $A$, the state vectors satisfy
\[
    v_i(HB)=A^i x,
    \qquad
    i=0,\ldots,d-1.
\]
Therefore
\[
    \det(HB)
    =
    \det(x,Ax,\ldots,A^{d-1}x)
    =
    f_A(x_1,\ldots,x_d).
\qedhere
\]
\end{proof}

\begin{corollary}\label{cor:tile-independence}
The value $\det(HB)$ depends only on $A$ and on the sequence $x$, not on the particular choice of a head whose perfect matching sequence is $x$ and polyomino wug-tile
for $A$.
\end{corollary}

\begin{proof}
This follows from Theorem~\ref{MD-det}, since $f_A(v)$ is determined
entirely by $A$.
\end{proof}

\begin{remark}
Besides canonical polyominoes, we describe in
Subsection~\ref{Recurrence relations for sequences and corresponding
wug-graphs} another important class of polyomino wug-tiles, of rank $d$,
arising from products of lower companion matrices. These are generally
non-canonical but often more natural for applications, as they connect to
the theory of multidimensional continued fractions.
\end{remark}

The definitions of the geometric and algebraic Markov numbers (Definition~\ref{alg-geom}) immediately yield the following consequence.

\begin{corollary}[\textbf{Markov numbers as wug-snake determinants}]\label{alg-geo-Markov-det-wug}
Let $A\in\Mat(d,\z)$, and let $B$ be a polyomino wug-tile for $A$.
\begin{enumerate}[label=\textup{(\roman*)}]
    \item There exists a head $H$ of length $d$ such that
    \[
        m(A)=|\det(HB)|.
    \]
    \item If $H_0$ is the canonical realisation of $(0,\ldots,0,1)$, then
    \[
        \hat m(A)=|\det(H_0B)|.
    \]
\end{enumerate}
\end{corollary}

\begin{proof}
Statement (ii) follows immediately from Theorem~\ref{MD-det}:
\[
    \hat m(A)
    =
    |f_A(0,\ldots,0,1)|
    =
    |\det(H_0B)|.
\]
For (i), note that $A$ has integer entries, so $f_A$ takes integer values
on $\z^d$. The set
$\{|f_A(p)| : p\in\z^d\setminus\{0\}\}$
is a nonempty subset of $\z_{\geq 0}$, and its infimum is attained at some
$p=(x_1,\ldots,x_d)$. Let $H$ be the canonical realisation of
$(x_1,\ldots,x_d)$. Theorem~\ref{MD-det} gives
\[
    m(A)=|f_A(p)|=|\det(HB)|.
\qedhere
\]
\end{proof}

\begin{example}\label{exM111M1015}
Consider the matrix
\[
M=
\begin{pmatrix}
2 & 1 & 3\\
3 & 2 & 4\\
6 & 4 & 9
\end{pmatrix}.
\]
We illustrate Corollary~\ref{alg-geo-Markov-det-wug}\,(ii) by computing
$\hat m(M)$. Attaching two copies of a polyomino $B(M)$ to
the head $H_0$ representing $(0,0,1)$ gives the following wug-snake graph:
\[
\includegraphics[height=2.9cm]{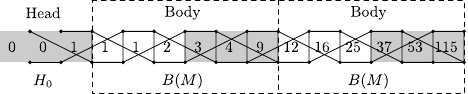}
\]
All edge weights equal $1$. The state vectors are
\[
    v_0 = \begin{pmatrix}0\\0\\1\end{pmatrix},\qquad
    v_1 = \begin{pmatrix}3\\4\\9\end{pmatrix},\qquad
    v_2 = \begin{pmatrix}37\\53\\115\end{pmatrix},
\]
and the wug-snake determinant is
\[
\det(H_0B(M))=
\det
\begin{pmatrix}
0 & 3 & 37\\
0 & 4 & 53\\
1 & 9 & 115
\end{pmatrix}
=11.
\]
For comparison, the full Markov-Davenport form of $M$ is
\[
f_M = 18x^3 + 18x^2y + 12x^2z + 26xy^2 - 22xyz - 14xz^2
     + 14y^3 + 18y^2z - 28yz^2 + 11z^3,
\]
confirming that $\hat m(M)=|f_M(0,0,1)|=11$. Moreover, in this particular example, we have $m(M)=\hat{m}(M)=11$. The details to arrive $m(M)=11$ are quite involved and we omit for this paper.
\end{example}

\subsection{Linear recurrences and lower companion matrices}
\label{Linear recursions and lower companion matrices}

We now recall the matrix formulation of linear recurrences.

\begin{definition}
Let $N\in\z_{>0}\cup\{\infty\}$. A sequence $(x_n)_{n=1}^N$ is
\textit{defined by a linear recurrence system of order $d$} if
\begin{itemize}
    \item its initial terms $x_1,\ldots,x_d$ are prescribed, and
    \item for $t=d,\ldots,N-1$,
    \[
        x_{t+1}
        =
        \sum_{i=1}^d c_{i,t}\,x_{t+1-i}.
    \]
\end{itemize}
The coefficients $c_{i,t}$ are allowed to depend on $t$.
\end{definition}

\begin{remark}
A sequence can satisfy more than one recurrence of a given order. For
instance, the sequence $(1,1,3)$ can be obtained from $x_1=x_2=1$ either
by $x_3=x_1+2x_2$ or by $x_3=2x_1+x_2$.
\end{remark}

%\begin{proposition}
%Consider a sequence $(x_n)_{n=1}^N$ and fix the number of initial terms $s$. Then the space of all possible iteration coefficients describing $(x_n)_{n=1}^N$ is an affine space.
%\end{proposition}
%
%\begin{proof}
%Indeed, let $(a_{k,t})$ and $(b_{k,t})$ be two sets of iteration coefficients defining the sequence $(x_n)_{n=1}^N$. Then the convex combination of these iteration coefficients,
%\[
%t(a_{k,t}) + (1-t)(b_{k,t}),
%\]
%defines $(x_n)_{n=1}^N$ as well for any real $t$.
%\end{proof}

%\begin{example}
%Consider the sequence $x_n = 2+(-1)^n$, which yields $(1,3,1,3,\ldots)$, and let $s=4$. Then
%\[
%x_n = x_{n-2} \qquad \text{and} \qquad x_n = x_{n-4}
%\]
%are two possible linear recurrence relations for $(x_n)$ for $n > 4$. Hence, for every real $t$, we also have the recurrence relation:
%\[
%x_n = t x_{n-2} + (1-t) x_{n-4}.
%\]
%\end{example}

\begin{definition}[Lower companion matrix]\label{def: lower companion matrix}
For $a_1,\ldots,a_d\in\z$, the \textit{lower companion matrix} $M_{a_1,\ldots,a_d} \in \Mat(d,\z)$ is given by
\[
M_{a_1,\ldots,a_d}
=
\begin{pmatrix}
0 & 1 & 0 & \cdots & 0\\
0 & 0 & 1 & \cdots & 0\\
\vdots & \vdots & \ddots & \ddots & \vdots\\
0 & 0 & \cdots & 0 & 1\\
a_1 & a_2 & \cdots & a_{d-1} & a_d
\end{pmatrix}.
\]
These matrices are also known as transposed Frobenius companion matrices.
\end{definition}

The connection between recurrences and companion matrices is the following.
If a sequence $(x_n)$ satisfies $x_{t+1}=\sum_{i=1}^d c_{i,t}\,x_{t+1-i}$,
then
\[
\begin{pmatrix}
x_{t-d+2}\\
x_{t-d+3}\\
\vdots\\
x_{t+1}
\end{pmatrix}
=
M_{c_{d,t},\ldots,c_{1,t}}
\begin{pmatrix}
x_{t-d+1}\\
x_{t-d+2}\\
\vdots\\
x_t
\end{pmatrix},
\qquad t=d,\ldots,N-1.
\]
Thus one step of a recurrence of order $d$ is encoded by multiplication by a
lower companion matrix. This motivates the following definitions.

\begin{definition}\label{Frob-to-E}
Let $d,N\in\z_{>0}$ with $d<N$, or let $N=\infty$. Let
\[
    \M
    =
    \bigl(M_{c_{d,t},\ldots,c_{1,t}}\bigr)_{t=d}^{N-1}
\]
be a sequence of lower companion matrices. The \textit{associated linear
recurrence system} is
\[
    x_{t+1}
    =
    \sum_{i=1}^{d}c_{i,t}\,x_{t+1-i},
    \qquad
    t=d,\ldots,N-1.
\]
We denote this system by $E(\M)$.
\end{definition}

%This observation provides a direct link between recurrence sequences
%, and multidimensional continued fractions, commonly referred to as \textit{subtractive algorithms} (see Subsection~\ref{Cyclic subtractive algorithms}).

\subsection{Recurrences and their wug-snake graphs}
\label{Recurrence relations for sequences and corresponding wug-graphs}

We now show that every linear recurrence system can be realised by a
wug-snake graph whose perfect matching sequence coincides with the given
recurrence sequence. The construction separates initial data from recurrence
data: the initial terms are encoded by the head, while the recurrence
coefficients determine the body.

\subsubsection{Heads for sequences}
Recall for a linear recurrence system of order $d$, the initial terms $x_1,\ldots,x_d$ are prescribed. Thus we can use, for instance, the canonical realisation of a sequence constructed in Subsection \ref{subsection: Canonical construction} for the initial terms.

\begin{remark}
Since the initial terms are not governed by the recurrence, the head is not
unique. In contrast, once a head is fixed, the body is completely determined
by the recurrence coefficients.
\end{remark}

\subsubsection{Construction of the body}

\begin{definition}\label{Perfect matching of a sequence -- definition}
Let $(x_n)_{n=1}^N$ be defined by the linear recurrence system of order $d$, with
\[
    x_{t+1}
    =
    \sum_{r=1}^{d} c_{r,t}\,x_{t+1-r},
    \qquad
    t=d,\ldots,N-1,
\]
and choose a head $H\in\Head(x_1,\ldots,x_d)$. The \emph{wug-snake graph associated with the recurrence system and the
head \(H\)} is the graph \(K=HB\) whose weight matrix is defined as follows:
\begin{itemize}
    \item on columns $1,\ldots,d$, the weight matrix agrees with that of
    $H$;
    \item for $j>d$ (corresponding to $t=j-1$),
    \[
        w_{i,j}
        =
        \begin{cases}
        c_{j+1-i,\,j-1}, & j+1-d\leq i\leq j,\\
        1,               & i=j+1,\\
        0,               & \text{otherwise}.
        \end{cases}
    \]
\end{itemize}
The subgraph $B$ is called the \textit{body associated with the recurrence
system}.
\end{definition}

\begin{theorem}\label{Perfect matching of a sequence -- theorem}
Let \(K=HB\) be the wug-snake graph constructed above. Then
\[
    \mu(K_j)=x_j,
    \qquad
    j=1,\ldots,N.
\]
Thus, the perfect matching
sequence of \(K\) is $(x_1,x_2,\ldots,x_N).$

Moreover, $\operatorname{rank}(K)\leq d.$ If $c_{d,t}\neq0$ for some \(t\in\{d,\ldots,N-1\}\), then $\operatorname{rank}(K)=d.$
\end{theorem}

\begin{proof}
For \(1\leq j\leq d\), the equality $\mu(K_j)=x_j$ holds by the choice of the head \(H\).

We proceed by induction on \(j>d\). Set \(t=j-1\). By
Theorem~\ref{theorem:wug-det}, $\mu(K_j)
    =
    \sum_{i=1}^{j}w_{i,j}\mu(K_{i-1}).$ By construction, the only possibly nonzero entries of column \(j\) in
rows \(1,\ldots,j\) are $w_{j+1-r,j}=c_{r,j-1}$ for $r=1,\ldots,d.$
Therefore, using the induction hypothesis,
\[
    \mu(K_j)
    =
    \sum_{r=1}^{d}
    c_{r,j-1}\mu(K_{j-r})
    =
    \sum_{r=1}^{d}
    c_{r,t}x_{t+1-r}
    =
    x_{t+1}
    =
    x_j.    
\]
This proves the assertion about the perfect matching sequence.

It remains to consider the rank. Every nonzero entry \(w_{i,j}\) with
\(i\leq j\) satisfies $ j-i\leq d-1.$ Indeed, this is automatic in the \(d\times d\) head, while for \(j>d\)
the construction gives \(i\geq j+1-d\). Hence $\operatorname{rank}(K)\leq d.$

If \(c_{d,t}\neq0\), then $w_{t+2-d,t+1}=c_{d,t}\neq0.$
This entry lies on the \((d-1)\)-st superdiagonal, since $(t+1)-(t+2-d)=d-1.$
Consequently, \(\operatorname{rank}(K)\geq d\), and therefore $\operatorname{rank}(K)=d.$
\end{proof}

\begin{example}\label{ex:sequence-194}
Consider the sequence
\[
    0,\;1,\;1,\;2,\;3,\;5,\;13,\;31,\;44,\;75,\;194.
\]
We take initial terms $x_1=0$, $x_2=1$. The remaining terms are defined by
\[
\renewcommand{\arraystretch}{1.2}
\begin{array}{llll}
x_3=x_1+x_2, &
x_4=x_2+x_3, &
x_5=x_3+x_4, &
x_6=x_4+x_5, \\
x_7=x_5+2x_6, &
x_8=x_6+2x_7, &
x_9=x_7+x_8, &
x_{10}=x_8+x_9, \\
\multicolumn{2}{l}{x_{11}=x_9+2x_{10}.}
\end{array}
\]
Two possible heads realising $(0,1)$ are shown below:
\[
\begin{array}{l}
\includegraphics[height=1.2cm]{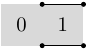}
\end{array}
\qquad\text{and}\qquad
\begin{array}{l}
\includegraphics[height=1.2cm]{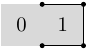}
\end{array}.
\]
Choosing the left head, the resulting wug-snake graph is
\[
\includegraphics[height=1.2cm]{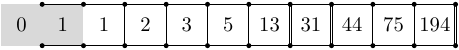}
\]
The grey part is the head; the white part is the body. All edge weights are
$1$ or $2$, as indicated.

The head has length $2$, the body has length $9$, and the full graph has
length $11$. The rank is $2$ (the recurrence has order $2$, which coincides with the rank). The neck-length
is~$1$.

The weight matrix (first $5$ columns displayed) is
\[
{\small
\begin{pmatrix}
0&1&0&0&0&\cdots\\
1&0&1&0&0&\cdots\\
0&1&1&1&0&\cdots\\
0&0&1&1&1&\cdots\\
0&0&0&1&1&\cdots\\
\vdots&\vdots&\vdots&\vdots&\vdots&\ddots
\end{pmatrix}.}
\]
\end{example}

\subsection{Lower companion matrices decomposition}
\label{Lower companion matrices decomposition}

In this subsection we study representations of matrices in $\Mat(m,\z)$ as
products of lower companion matrices. This property will be used in Subsection \ref{A few words on the rank of wug-tiles representing matrices}.

Let us introduce the following notations.
\begin{definition}\label{LC-definition}
Denote by 
\begin{itemize}
\item
$\LC(m,\z)$ the semigroup generated by all lower companion
matrices of size $m$ over $\z$;

\item $\LC_{\geq0}(m,\z)$ the subsemigroup generated by those whose last row has nonnegative entries;

\item
$\GLC(m,\z)$ the set of all invertible matrices in $\LC(m,\z)$;

\item 
$\GLC_{\geq0}(m,\z)$ the semigroup of all invertible matrices in $\LC_{\geq0}(m,\z)$.

\end{itemize}
\end{definition}

In fact, $\GLC(m,\z)$ is a group. This is justified by the following remark.

\begin{remark}\label{Remark-inverse}
Note that if $M\in \LC(m,\z)$ and $M$ is invertible, then $M^{-1}\in \LC(m,\z)$. Therefore, $\GLC(m,\z)$ is a group.
Indeed, the claim is true for the generators:
$$
M^{-1}_{\pm 1,a_2,\ldots a_m}=
(M_{1,0,\ldots, 0})^{m-1}
M_{\pm 1,\mp a_2,\ldots \mp a_m}
(M_{1,0,\ldots, 0})^{m-1}
\in \LC(m,\z),
$$
here we take either upper signs or lower signs. Now let \(M\in\LC(m,\z)\) be invertible then \(\det M=\pm1\). Since \(M\) is a product of lower
companion matrices and \(\det M=\pm1\), it follows that every factor in such a product is also invertible and the above formula shows that the inverse of each factor is in $\LC(m,\z)$. Hence, $M^{-1}$ is again a product of lower companion matrices. Therefore, $M^{-1}\in \LC(m,\z).$ 
\end{remark}

We have the following structure theorem.

\begin{theorem}\label{thm:lc-mat-glc-ll}
Let $m\ge 2$. Then
\begin{enumerate}[label=\textup{(\roman*)}]
    \item $\LC(m,\z)=\Mat(m,\z)$;
    \item $\GLC(m,\z)=\GL(m,\z)$.
\end{enumerate}
\end{theorem}

\begin{proof}
By the Smith normal forms theory for integer matrices, for any matrix $A\in \Mat(m,\z)$, there exist $U,V \in \GL(m,\z)$ and a diagonal matrix $D$ such that $UAV=D$. Moreover, $U$ and $V$ are products of elementary unimodular matrices. We
use the standard fact that such matrices are generated by permutation
matrices, Chevalley matrices \(E_{i,j}(\lambda)=I+\lambda e_{i,j}\), with
\(i\neq j\) and \(\lambda\in\z\) and diagonal sign-change matrices.

\vspace{2mm}

It suffices to show that every diagonal matrix belongs to $\LC(m,\z)$, and
that permutation matrices, Chevalley matrices, and diagonal sign-change
matrices belong to $\GLC(m,\z)$.
We prove this by discussing the following cases.
\begin{itemize}

\vspace{2mm}

\item
{\bf Diagonal matrices.} Let $D=\operatorname{diag}(\lambda_1,\ldots,\lambda_m)$. Consider the
following sequence 
$$
x_1,\ldots, x_m,
\quad
\lambda_1 x_1,\ldots, \lambda_m x_m.
$$
Its first $m$ terms form the vector $X=(x_1,\ldots,x_m)^T$, and its last
$m$ terms form $DX$. This sequence satisfies a recurrence system of order
$m$: at the $j$-th step one obtains $\lambda_jx_j$ from the oldest entry
$x_j$ in the current state vector. Hence, $ D\in\LC(m,\z).$ 

If $D\in\GL(m,\z)$, then each $\lambda_j=\pm1$, and the corresponding
recurrence steps are invertible. Hence, $D\in\GLC(m,\z).$

In particular, all diagonal sign-change matrices belong to $\GLC(m,\z)$.
\vspace{1mm}

\item
{\bf Coordinate cycles.} The lower companion matrix $M_{1,0,\ldots, 0}$ acts by the coordinate cycle
\[
    (x_1,\ldots,x_m)^T
    \longmapsto
    (x_2,\ldots,x_m,x_1)^T.
\]
It is invertible, and therefore belongs to $\GLC(m,\z)$.

\vspace{1mm}

\item
{\bf Coordinate transpositions.} Let $T_{1,2}$ be the transposition matrix swapping the first two basis
vectors. Consider the sequence
$$
x_1, x_2, \ldots, x_m,
\quad
x_1+\sum\limits_{i=3}^m x_i,
\quad
x_2,x_1,x_3,\ldots, x_m.
$$
Indeed, the first $m$ terms form $X$, while the last $m$ terms form
$T_{1,2}X$. The displayed sequence satisfies an order-$m$ recurrence whose
steps are invertible lower companion steps. Hence, $T_{1,2}\in\GLC(m,\z).$

\vspace{1mm}
\item
{\bf Chevalley matrices.} Let $E_{1,2}(\lambda)$ be the matrix obtained from the identity
matrix by adding $\lambda$ in position $(1,2)$. Thus
\[
    E_{1,2}(\lambda)X
    =
    (x_1+\lambda x_2,x_2,\ldots,x_m)^T.
\]
Consider the sequence
$$
x_1, x_2, \ldots, x_m,
\quad
x_1+\lambda x_2,x_2,\ldots, x_m.
$$
Again the recurrence steps are invertible lower companion steps, and hence, $E_{1,2}(\lambda)\in \GLC(m,\z)$. 
\end{itemize}

\vspace{2mm}

Since the coordinate cycle $M_{1,0,\ldots,0}$ and the transposition
$T_{1,2}$ both belong to $\GLC(m,\z)$, the group $\GLC(m,\z)$ contains
all coordinate permutation matrices. Therefore, conjugating $E_{1,2}(\lambda)$ by permutation matrices
shows that every Chevalley matrix $E_{i,j}(\lambda)$ belongs to
$\GLC(m,\z)$. Hence every elementary
unimodular matrix belongs to $\GLC(m,\z)$.

Now let $A\in\Mat(m,\z)$. By Smith normal form, $A=U^{-1}DV^{-1},$ where $U,V$ are products of elementary unimodular matrices and $D$ is
diagonal. Now we know that $U,V \in \GLC(m,\z)$ and we have shown in Remark~\ref{Remark-inverse} that $ U^{-1},V^{-1}\in\GLC(m,\z)\subset\LC(m,\z),$ and that $D\in\LC(m,\z).$ Since $\LC(m,\z)$ is a semigroup, it follows that $A=U^{-1}DV^{-1}\in\LC(m,\z).$ Therefore, $\LC(m,\z)=\Mat(m,\z).$

If, in addition, $A\in\GL(m,\z)$, then in its Smith normal form the diagonal
matrix $D$ is unimodular. Hence, by the diagonal case,
$D\in\GLC(m,\z).$
Therefore, $A=U^{-1}DV^{-1}\in\GLC(m,\z),$
and consequently $\GLC(m,\z)=\GL(m,\z).$ This completes the proof.
\end{proof}

\subsection{Reduced matrices and LLS-sequences}
\label{Semigroup of reduced matrices}

In this subsection we briefly discuss the
semigroup $\GLC_{\geq0}(2,\z)$, which appears in Gauss reduction theory for
$\GL(2,\z)$. Its elements may be viewed as a discrete analogue of Jordan normal
forms for invertible integer matrices.

In dimension \(2\), the invertible nonnegative lower companion matrices are $\begin{psmallmatrix}
    0 & 1\\
    1 & q
    \end{psmallmatrix}$ for $q\in\z_{\geq0}$.

We shall focus on the positive subsemigroup generated by those $\begin{psmallmatrix}
    0 & 1\\
    1 & q
    \end{psmallmatrix}$ with
\(q\geq1\). This subsemigroup admits a particularly simple description in
terms of reduced matrices.

\begin{definition}\label{def-reduced-matrix}
A matrix
\[
    M=
    \begin{pmatrix}
    a & b\\
    c & d
    \end{pmatrix}
    \in \GL(2,\z)
\]
is called \textit{reduced} if $d \geq c\geq a \geq 0$ and $d \geq b \geq a \geq 0$.
\end{definition}

\begin{remark}
Here a matrix in \(\GL(2,\z)\) is called \emph{hyperbolic} if none of
its eigenvalues has absolute value \(1\). Every hyperbolic matrix in
\(\GL(2,\z)\) with positive trace is integrally conjugate to at least
one reduced matrix, and each integral conjugacy class contains only
finitely many reduced matrices. In other words, for such a matrix
\(A\), there exist finitely many reduced matrices of the form $U^{-1}AU$ for $U\in\GL(2,\z),$ and at least one such matrix exists.
\end{remark}

\begin{proposition}\label{PLLS-prop}
For every reduced matrix $M$, there exists a unique $n \geq 1$ and a unique finite sequence of
positive integers $(a_1,\ldots, a_n)$ such that 
\[
M=\prod_{i=1}^n
\begin{pmatrix}
0 & 1\\
1 & a_{n-i+1}
\end{pmatrix}.
\]
Conversely, any matrix in $\GL(2,\z)$ admitting such a positive companion factorization is reduced.
\end{proposition}

This is a classical statement; we refer the interested reader to~\cite{OK2022}.

\begin{definition}
The sequence $(a_1,\ldots,a_n)$ in Proposition~\ref{PLLS-prop} is called
the \textit{PLLS-sequence} of $M$. The bi-infinite periodic sequence with
period $(a_1,\ldots,a_n)$ is called the \textit{LLS-sequence} of $M$ and
is denoted by $\LLS(M)$.
\end{definition}

\begin{remark}
The ordering convention for the PLLS-sequence is chosen to agree with
matrix-vector multiplication: the entry $a_1$ corresponds to the
\emph{last} companion matrix factor in the product in
Proposition~\ref{PLLS-prop}.
\end{remark}

\begin{example}
Let
\[
    M=
    \begin{pmatrix}
    2 & 7\\
    5 & 17
    \end{pmatrix}
    =
    \begin{pmatrix}
    0 & 1\\
    1 & 2
    \end{pmatrix}
    \begin{pmatrix}
    0 & 1\\
    1 & 2
    \end{pmatrix}
    \begin{pmatrix}
    0 & 1\\
    1 & 3
    \end{pmatrix}.
\]
The PLLS-sequence of $M$ is $(3,2,2)$, and
$\LLS(M)=(\ldots,3,2,2,3,2,2,3,2,2,\ldots)$.
\end{example}

\begin{remark}
Since each factor $\bigl(\begin{smallmatrix}0&1\\1&a\end{smallmatrix}\bigr)$
is a unimodular companion matrix, Proposition~\ref{PLLS-prop} implies that
the set of reduced matrices is closed under multiplication. Hence reduced
matrices form a semigroup.
\end{remark}

We conclude this subsection with the following open problem.

\begin{problem}
Describe which square integer matrices have integer conjugates lying in
$\GLC_{\geq0}(m,\z)$ or in $\LC_{\geq0}(m,\z)$.
\end{problem}

To the best of our knowledge, even the case $m=3$ has not been fully
explored.

\subsection{Ranks of wug-tiles representing matrices}
\label{A few words on the rank of wug-tiles representing matrices}
Following the structure theorem of matrix decomposition in Subsection \ref{Lower companion matrices decomposition}, we start with the following construction.

\begin{definition}\label{def:body-from-decomposition}
Let $A\in\LC(m,\z)$, and let $A=M_k\cdots M_1$ be a decomposition into
lower companion matrices. Set $\M=(M_1,\ldots,M_k)$. The body of the
wug-snake graph associated with the recurrence system $E(\M)$ is denoted
by $B_{\M}(A)$.
\end{definition}
We present in the following example two realizations of a matrix in $\Mat(2,\z)$.
\begin{example}
Consider the following matrix and its decomposition into product of lower companion matrices
$$
M=
\begin{pmatrix}
    3 & 2\\
    4 & 3
    \end{pmatrix}
    =
    (M_{-1,2}M_{1,1})^2.
$$
Fix the canonical head \(H_{(x,y)}\) with state vector
\((x,y)^T\). The full wug-snake graph
\(H_{(x,y)}B_{\mathrm{can}}(M)\) associated with the canonical tile has
weight matrix $w_1$. Since its nonzero upper-triangular entries have width at most \(3\), and
the entry in position \((2,4)\) is nonzero, this realisation has rank
\(3=2m-1\), where \(m=2\). 

On the other hand, the decomposition $(M_{-1,2}M_{1,1})^2$ gives a rank-\(2\) realisation. Let $\M
=
\bigl(
M_{1,1},M_{-1,2},M_{1,1},M_{-1,2}
\bigr)$. The full wug-snake graph
\(H_{(x,y)}B_{\M}(M)\) has weight matrix $w_2$:
\[
w_1 = \begin{pmatrix}
x&y&0&0\\
1&0&3&2\\
0&1&4&3\\
0&0&1&0
\end{pmatrix};
\qquad\qquad
w_2=\begin{pmatrix}
x&y&0&0&0&0\\
1&0&1&0&0&0\\
0&1&1&-1&0&0\\
0&0&1&2&1&0\\
0&0&0&1&1&-1\\
0&0&0&0&1&2
\end{pmatrix}.
\]
Thus \(M\) admits both a canonical rank-\(3\) realisation and a
rank-\(2\) realisation arising from its lower companion factorisation.

We return to this factorisation in
Subsection~\ref{Classical snake graphs as 3x3 systems}.

\end{example}

For \(A\in\Mat(m,\z)\), define
\[
    \rho(A)
    =
    \min\{
        \operatorname{rank}(B):
        B\in\mathcal T_m(A)
        \text{ and }B\text{ is nonempty}
    \}.
\]

\begin{theorem}\label{transFrob-tiles}
Let \(m\geq2\) and \(A\in\Mat(m,\mathbb Z)\). Then
\begin{enumerate}[label=\textup{(\roman*)}]
    \item \(A\) admits a polyomino wug-tile representation of rank at
    most \(m\).

    \item More precisely, for every decomposition $A=M_k\cdots M_1$
    into lower companion matrices, the body \(B_{\M}(A)\) is a
    polyomino wug-tile for \(A\) of rank at most \(m\).

    \item If $Ae_1\neq0$ for $e_1=(1,0,\ldots,0)^T$, then $ \rho(A)=m.$
\end{enumerate}
\end{theorem}

\begin{proof} By
Theorem~\ref{thm:lc-mat-glc-ll}\textup{(i)}, we have $\Mat(m,\z)=\LC(m,\z)$. Hence every matrix $A\in\Mat(m,\z)$ admits a decomposition
$A=M_k\cdots M_1$ into lower companion matrices.

The recurrence system $E(\M)$ acts on its state vector by successive
multiplication by $M_1,\ldots,M_k$. Hence the total action is multiplication by $M_k\cdots M_1=A$. Therefore $B_{\M}(A)$ is a polyomino wug-tile for $A$. Since
\(E(\M)\) has order \(m\), Theorem~\ref{Perfect matching of a sequence -- theorem}
shows that $\operatorname{rank}(B_{\M}(A))\leq m$. This proves
\textup{(i)} and \textup{(ii)}.

It remains to prove \textup{(iii)}. Let \(B\) be a nonempty polyomino
wug-tile for \(A\), and suppose for contradiction that $\operatorname{rank}(B)=r<m.$ Let $(x_1,\ldots,x_m)^T$ be the state vector before \(B\) is attached. By the definition of
rank and the perfect-matching recurrence, the first new term produced
by the body depends only on $x_{m+1-r},\ldots,x_m.$ In particular, it is independent of \(x_1\). Every subsequent new term
is a linear combination of terms that are themselves independent of
\(x_1\). Since \(B\) is nonempty, the entry \(x_1\) is removed from the state vector after the first recurrence step and cannot influence any later term.

It follows that the final state vector is independent of \(x_1\).
Equivalently, $Ae_1=0$, contrary to the hypothesis. Hence every nonempty polyomino wug-tile
for \(A\) has rank at least \(m\). Part \textup{(i)} provides one of
rank at most \(m\), so $\rho(A)=m.$
\end{proof}

\begin{remark}
By Proposition~\ref{rank-canonic}, the canonical body associated with
\(A\in\Mat(m,\z)\) has rank at most \(2m-1\). Theorem~\ref{transFrob-tiles}
shows that every integer matrix admits another polyomino wug-tile
representation of rank at most \(m\). If \(Ae_1\neq0\), this bound is
optimal among nonempty polyomino wug-tiles.
\end{remark}

We conclude this subsection with some examples of matrices in $\Mat(2,\z)$ that admit 
rank-$2$ polyomino wug-tile representations.

\begin{example}
Consider the second Cohn matrix
$
A=
\begin{pmatrix}
    3 & 2 \\
    4 & 3
\end{pmatrix}.
$ Write $M_{1,q} = \begin{pmatrix}
        0&1\\1&q
    \end{pmatrix}$, we have
\[
    A= M_{1,1}M_{1,2}M_{1,1}M_{1,0}
\]
All these factors have nonnegative entries. Therefore the matrix belongs to $\GLC_{\geq0}(2,\z)$ and admits a rank-2 tile with nonnegative weights. But $A$ is not reduced in the sense of Definition \ref{def-reduced-matrix}, because it is not in the smaller semigroup $\langle M_{1,q}: q \geq 1\rangle$.
\end{example}

\begin{example}
Consider
\[
    A=M_{-1,2}
    =
    \begin{pmatrix}
        0&1\\
        -1&2
    \end{pmatrix}.
\]
Since \(A\) is itself a lower companion matrix, it admits a rank-\(2\)
polyomino wug-tile representation. Moreover, $ Ae_1= (0,-1)^{T}\neq0,$
so Theorem~\ref{transFrob-tiles} implies that its minimal possible rank
is \(2\).

On the other hand, a polyomino wug-tile with nonnegative weights has an
entrywise nonnegative snake operator: this follows inductively from the
perfect-matching recurrence. Since \(A\) has a negative entry, it cannot
be represented by a polyomino wug-tile with nonnegative weights.
\end{example}
%\begin{remark}
%Theorem~\ref{MD-det} and %Corollary~\ref{alg-geo-Markov-det-wug} %apply to
%every matrix in $\LC(n,\z)$: its Markov-Davenport values and Markov numbers
%are expressible as wug-snake determinants.
%\end{remark}

%\begin{example}
%The matrix $M$ of Example~\ref{exM111M1015} belongs to $\LC(3,\z)$ via the
%decomposition $M=M_{1,1,1}\,M_{1,0,1}^{\,5}$. The corresponding body
%$B_{\M}(M)$ is the polyomino wug-tile used in that example.
%\end{example}

\subsection{A remark on CW-complexes and face matchings}
\label{Wug-snake graphs for CW-complexes}

We briefly indicate a possible geometric extension of the wug-snake
framework. The discussion is informal and intended to suggest a direction
for future work rather than to develop a complete theory.
\begin{definition}
Let $T$ be a three-dimensional CW-complex. A \textit{perfect
face-matching} of $T$ is a collection of pairwise disjoint two-dimensional
faces such that every vertex of $T$ belongs to exactly one chosen face. The
number of perfect face-matchings of $T$ is denoted by $\mu_2(T)$.
\end{definition}

The wug-snake construction can be extended to sequences of three-dimensional
boxes whose edges are connected by quadrilateral faces. For instance, a
recurrence of the form $x_n = a\,x_{n-1} + b\,x_{n-2} + c\,x_{n-3}$ admits
the following face-matching decomposition:
\[
    \begin{aligned}
    \vcenter{\hbox{\includegraphics[width=5cm]{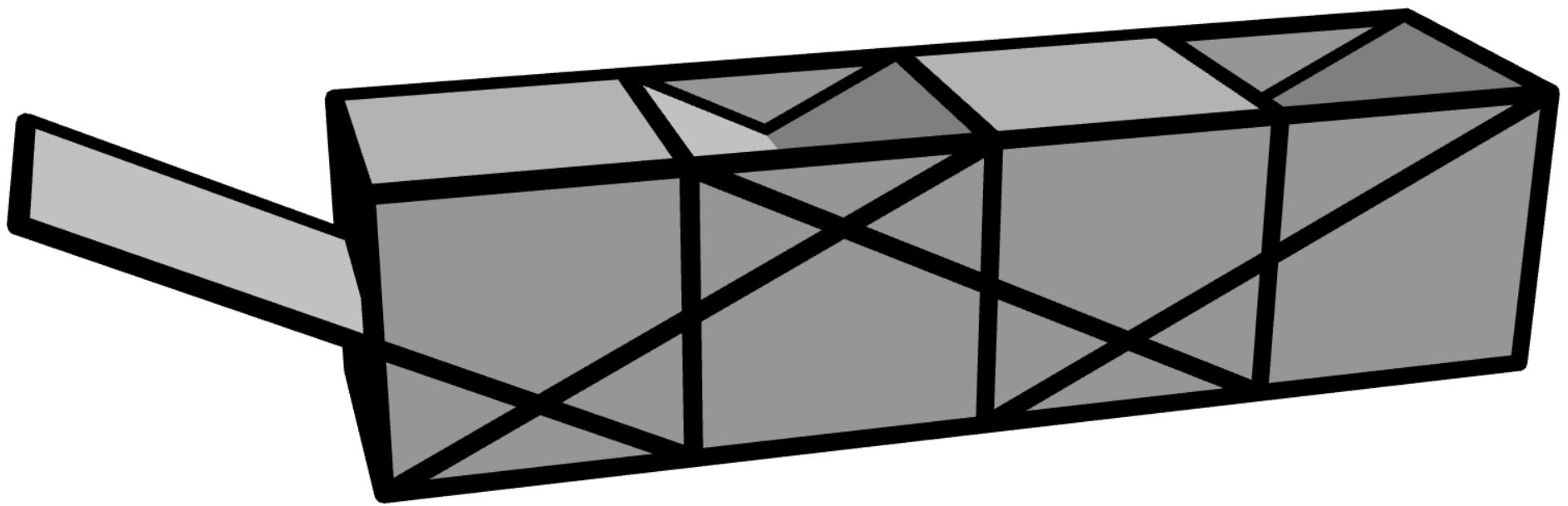}}}
    &=
    \vcenter{\hbox{\includegraphics[width=5cm]{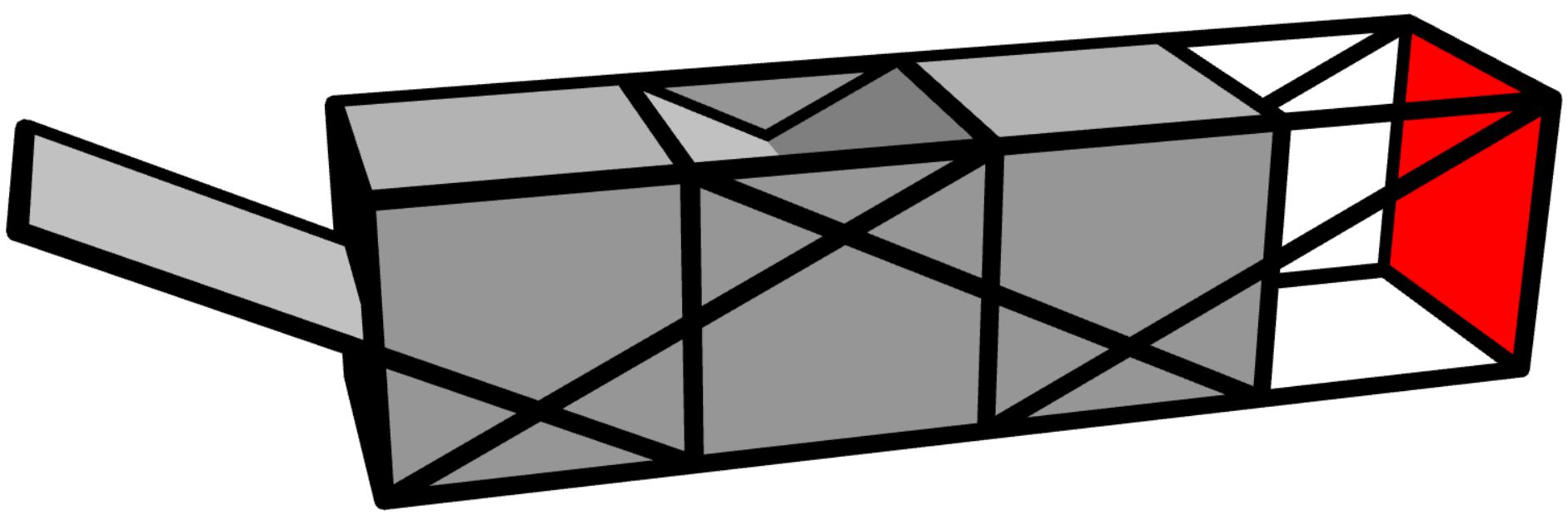}}} \\
    &\quad+
    \vcenter{\hbox{\includegraphics[width=5cm]{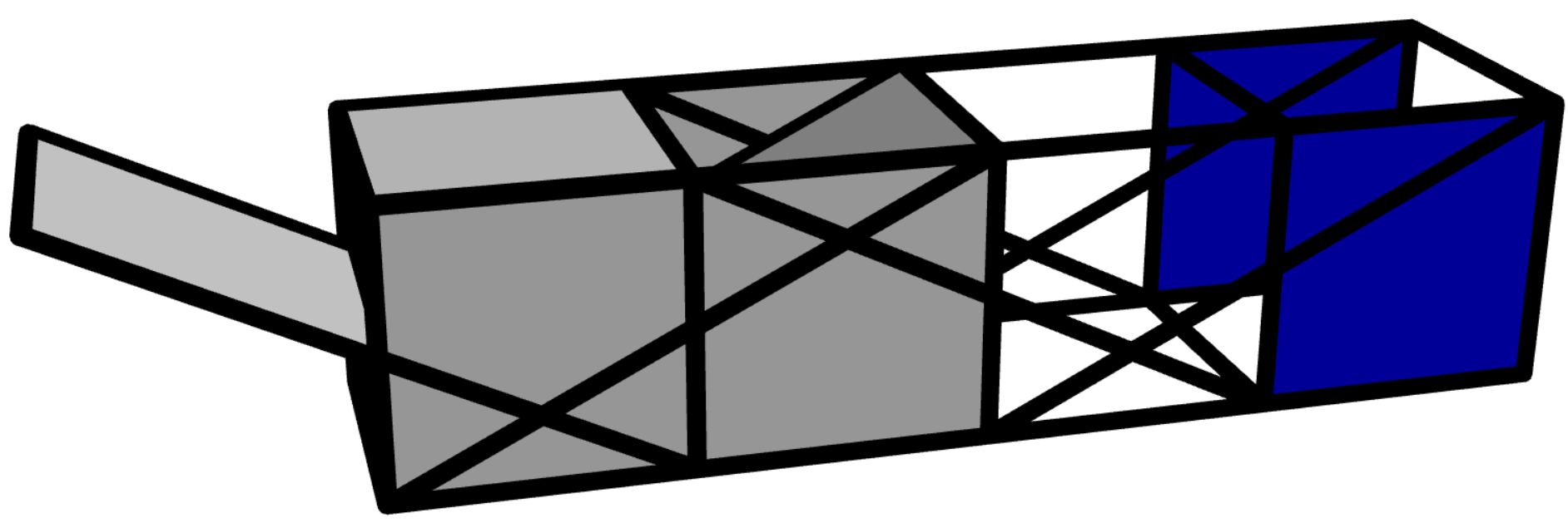}}} \\
    &\quad+
    \vcenter{\hbox{\includegraphics[width=5cm]{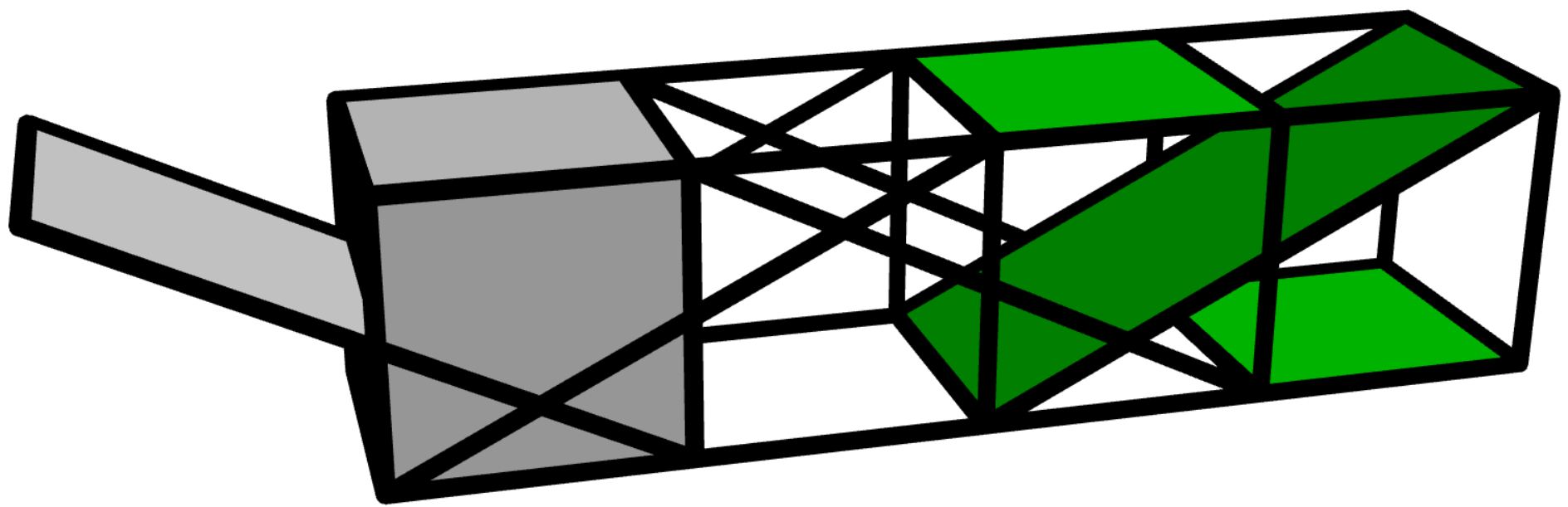}}}.
    \end{aligned}
\]
The coefficients $a$, $b$, $c$ appear as weights on the red, blue, and
green faces, respectively.
Namely, $a$ is the weight for the red face; $b$ and $1$ are weights on the blue faces;
$a$, $1$ and $1$ are weights on the green faces. 
The grey faces of the summands on the right hand side of the equation have the same weights as the corresponding faces on the left hand side. 

This CW-complex viewpoint does not produce new Markov numbers: the
resulting face-matching count reduces to the weighted perfect matching number
of the corresponding wug-snake graph:
\[
    \includegraphics[height=2.5cm]{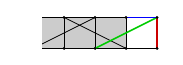}
\]
Nevertheless, the three-dimensional realisation may offer geometric insight
into lattice configurations that is not visible from the abstract graph
alone.

\begin{problem}
Develop a theory of perfect face-matchings in CW-complexes, relating their
enumeration to lattice geometry in $\r^n$.
\end{problem}

\section{Markov numbers of semigroups: definitions and examples}
\label{Markov numbers of semigroups: definitions and examples}

In this section we extend the classical Frobenius map on Markov numbers
to finitely generated matrix semigroups. The basic input is a subdivision
scheme, which plays the role of the Farey tree, together with a choice of
matrix generators. This produces a Cohn-type map from the vertices of the
resulting tessellation to the semigroup. Composing this map with the
geometric or algebraic Markov number gives the corresponding Frobenius map
for the semigroup.

\vspace{2mm}

After providing general definitions and notation, we discuss several examples of old and new wug-snake graph constructions. We
first recover the classical Markov numbers from semigroups of $2\times2$
matrices and from their wug-snake realisations. We then consider semigroups
generated by matrices with PLLS-sequences $(a,a)$ and $(b,b)$. Finally
we study a three-generator example whose wug-snake graphs admit natural
embeddings into $\z^3$.

\subsection{Basic definitions}
\label{Markov numbers of semigroups: basic definitions}
\subsubsection{Farey sums of rational vectors}
\label{Farey addition of rational vectors}

Let $v\in\q^n$. A representation
\[
    v=
    \left(
    \frac{p_1}{q},\frac{p_2}{q},\ldots,\frac{p_n}{q}
    \right)
\]
is called a \textit{regular fractional representation} of $v$ if
$p_1,\ldots,p_n$ are integers, $q$ is a positive integer, and $\gcd(p_1,\ldots,p_n,q)=1$.

\begin{definition}
Let $v_1,\ldots,v_k\in\q^n$, and write their regular fractional
representations as
\[
    v_i=
    \left(
    \frac{p_{1,i}}{q_i},
    \frac{p_{2,i}}{q_i},
    \ldots,
    \frac{p_{n,i}}{q_i}
    \right),
    \qquad i=1,\ldots,k.
\]
For nonnegative integers $\lambda_1,\ldots,\lambda_k$ that are not all zero, the
\textit{Farey sum with multiplicities} \(\lambda_1,\ldots,\lambda_k\) is defined by
\[
    \bigoplus_{i=1}^k \lambda_i v_i
    =
    \left(
    \frac{\sum_{i=1}^k \lambda_i p_{1,i}}
         {\sum_{i=1}^k \lambda_i q_i},
    \frac{\sum_{i=1}^k \lambda_i p_{2,i}}
         {\sum_{i=1}^k \lambda_i q_i},
    \ldots,
    \frac{\sum_{i=1}^k \lambda_i p_{n,i}}
         {\sum_{i=1}^k \lambda_i q_i}
    \right).
\]
\end{definition}
\begin{remark}
The denominator in the preceding definition is positive because
\(q_i>0\) and the multiplicities are not all zero. Moreover,
\[
    \bigoplus_{i=1}^k\lambda_i v_i
    =
    \frac{\sum_{i=1}^k\lambda_iq_i v_i}
         {\sum_{i=1}^k\lambda_iq_i},
\]
so the Farey sum belongs to the convex hull of
\(v_1,\ldots,v_k\).
\end{remark}

\subsubsection{Affine subdivision schemes and tessellations}

Let \(\Aff_{n+1}(\mathbb Q^n)\) denote the set of ordered rational
\(n\)-simplices, that is, ordered \((n+1)\)-tuples of affinely
independent points in \(\mathbb Q^n\).

We say that triangulations of
two simplices $\sigma_1$ and $\sigma_2$ with ordered vertices 
are {\it combinatorially equivalent}
if there exist a homeomorphism of $\sigma_1$ to $\sigma_2$ which preserves the structure of faces and the order of vertices and which maps the triangulation of $\sigma_1$ to the triangulation of $\sigma_2$.

\begin{definition}
Let $k$ be a positive integer. A map
\[
    \Lambda:
    \Aff_{n+1}(\mathbb Q^n)
    \longrightarrow
    \bigl(\Aff_{n+1}(\mathbb Q^n)\bigr)^k
\]
is called a \emph{subdivision scheme into \(k\) simplices} if, for every
simplex $\sigma$, the $k$ simplices in $\Lambda(\sigma)$ form a
triangulation of $\sigma$, and all such triangulations are
combinatorially equivalent.
\\
By $|V(\Lambda)|$ we denote the number of vertices in the
subdivision of a simplex under $\Lambda$.
\end{definition}

\begin{definition}
Let $\Lambda$ be a subdivision scheme. A subset
$W\subset\Aff_{n+1}(\mathbb Q^n)$ is called \textit{aligned with}
$\Lambda$ if, for every simplex $s\in W$, all simplices in
$\Lambda(s)$ also belong to $W$. We denote the set of all such subsets
by $\Al(\Lambda)$.
\end{definition}

\begin{definition}
%[Affine subdivision scheme]
A subdivision scheme $\Lambda$ is called an
\emph{affine Farey subdivision scheme} if there exist
nonnegative integer vectors
\[
    (\lambda_{1,j},\ldots,\lambda_{n+1,j}),
    \qquad j=1,\ldots,|V(\Lambda)|,
\]
such that, for every simplex with vertices $v_1,\ldots,v_{n+1}$, the
vertices $w_1,\ldots,w_{|V(\Lambda)|}$ appearing in its subdivision are
given by
\[
    w_j=\bigoplus_{i=1}^{n+1}\lambda_{i,j}v_i,
    \qquad
    j=1,\ldots,|V(\Lambda)|.
\]
The matrix $(\lambda_{i,j})$ is called the \textit{subdivision matrix} of
$\Lambda$.
\end{definition}

\begin{remark}
Multiplying all entries of the subdivision matrix by a common positive
scalar does not change the corresponding subdivision rule. In what follows,
we assume that a subdivision matrix has been fixed for each affine
subdivision scheme.
\end{remark}

\begin{definition}
Let \(\sigma\) be an ordered rational simplex and let $\Lambda$ be an affine subdivision
scheme. The \textit{tessellation of $\sigma$ with respect to $\Lambda$}
is the smallest aligned collection of simplices containing $\sigma$:
\[
    \mathcal T_\Lambda(\sigma)
    =
    \bigcap_{\substack{W\in\Al(\Lambda)\\ \sigma\in W}} W.
\]
We denote by $V(\mathcal T_\Lambda(\sigma))$ the set of all vertices of
all simplices in this tessellation.
\end{definition}

\subsubsection{Examples of tessellations}

\begin{example}[Classical Farey tessellation]
The classical Farey tessellation starts with the segment $\sigma=[0,1]$.
The subdivision rule is
\[
    \Lambda([a,b])=([a,c],[c,b]),
    \qquad
    c=a\oplus b.
\]
The subdivision matrix, with vertices ordered as $(a,b,c)$, is
$
    \begin{pmatrix}
    1&0&1\\
    0&1&1
    \end{pmatrix}.
$
\end{example}

\begin{example}[Farey tessellations of a two-simplex]
\label{Farey tessellations examples}
Let $\sigma$ be the two-simplex with vertices $a,b,c$. We describe
three standard Farey-type subdivision rules.

\begin{itemize}
\item \textbf{Simultaneous Farey summation.}
Set
$
    d=a\oplus b\oplus c.
$
Then the simplex is subdivided into the three simplices
\[
    (a,b,d),\qquad (b,c,d),\qquad (c,a,d).
\]
The subdivision matrix, with vertices ordered as $(a,b,c,d)$, is
\[
    \begin{pmatrix}
    1&0&0&1\\
    0&1&0&1\\
    0&0&1&1
    \end{pmatrix}.
\]

\item \textbf{Pairwise Farey summation.}
Set
\[
    d=a\oplus b,\qquad e=a\oplus c,\qquad f=b\oplus c.
\]
The corresponding subdivision is
\[
    (a,d,e),\qquad (b,f,d),\qquad (c,e,f),\qquad (d,f,e).
\]
The subdivision matrix, with vertices ordered as $(a,b,c,d,e,f)$, is
\[
    \begin{pmatrix}
    1&0&0&1&1&0\\
    0&1&0&1&0&1\\
    0&0&1&0&1&1
    \end{pmatrix}.
\]

\item \textbf{Barycentric Farey summation.}
Set
\[
    d=a\oplus b,\qquad e=a\oplus c,\qquad f=b\oplus c,
    \qquad g=a\oplus b\oplus c.
\]
The corresponding subdivision is
\[
    (a,d,g),\quad (d,b,g),\quad (b,f,g),\quad
    (f,c,g),\quad (c,e,g),\quad (e,a,g).
\]
The subdivision matrix, with vertices ordered as $(a,b,c,d,e,f,g)$, is
\[
    \begin{pmatrix}
    1&0&0&1&1&0&1\\
    0&1&0&1&0&1&1\\
    0&0&1&0&1&1&1
    \end{pmatrix}.
\]
\end{itemize}

The three subdivisions are illustrated below:
\[
\begin{array}{c}
\includegraphics[width=2.5cm]{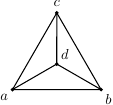}
\end{array}
\qquad
\begin{array}{c}
\includegraphics[width=2.5cm]{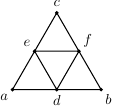}
\end{array}
\qquad
\begin{array}{c}
\includegraphics[width=2.5cm]{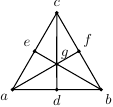}
\end{array}.
\]
\end{example}

\subsubsection{Cohn maps and Farey sets for semigroups}

Fix an ordered rational simplex $\sigma=(e_1,\ldots,e_{n+1})$ and an affine Farey subdivision scheme \(\Lambda\), with subdivision
matrix \(L=(\lambda_{i,j})\).

\begin{definition}
Let \(S=\langle g_1,\ldots,g_{n+1}\rangle\) be a semigroup with an
ordered generating set. The \textit{Cohn map}
\[
    \mathcal F_{g_1,\ldots,g_{n+1}}^\Lambda:
    V(\mathcal T_\Lambda(\sigma))\to S
\]
is defined recursively as follows. On the vertices $e_1,\ldots,e_{n+1}$
of the initial simplex, set
\[
    \mathcal F_{g_1,\ldots,g_{n+1}}^\Lambda(e_i)=g_i.
\]
If a new vertex $w_j$ is obtained from vertices $v_1,\ldots,v_{n+1}$ by
\[
    w_j=\bigoplus_{i=1}^{n+1}\lambda_{i,j}v_i,
\]
then define
\[
    \mathcal F_{g_1,\ldots,g_{n+1}}^\Lambda(w_j)
    =
    \prod_{i=n+1}^{1}
    \left(
    \mathcal F_{g_1,\ldots,g_{n+1}}^\Lambda(v_i)
    \right)^{\lambda_{i,j}},
\]
where the product is taken in the displayed order, from $i=n+1$ down to $i=1$.
\end{definition}

\begin{definition}
The \emph{Farey set} associated with the ordered generators
\((g_1,\ldots,g_{n+1})\), the initial simplex \(\sigma\), and the
subdivision scheme \(\Lambda\) is
\[
    \mathcal F_{g_1,\ldots,g_{n+1}}^\Lambda(\sigma)
    :=
    \mathcal F_{g_1,\ldots,g_{n+1}}^\Lambda
    \left(
        V\bigl(\mathcal T_\Lambda(\sigma)\bigr)
    \right).
\]
\end{definition}

\subsubsection{Geometric and algebraic Frobenius maps for semigroups}

Recall that $m(A)$ and $\hat m(A)$ denote the geometric and algebraic
Markov numbers of a matrix $A$, respectively.

\begin{definition}
Let $A_1,\ldots,A_{n+1}\in\Mat(d,\z)$, and let $\Lambda$ be an affine Farey subdivision scheme. The
\textit{geometric Frobenius map} associated with $\langle A_1,\ldots,A_{n+1}\rangle$ is
\[
    \phi_{A_1,\ldots,A_{n+1}}^\Lambda:
    V(\mathcal T_\Lambda(\sigma))\to\r
\]
defined by
\[
    \phi_{A_1,\ldots,A_{n+1}}^\Lambda(v)
    =
    m\left(
    \mathcal F_{A_1,\ldots,A_{n+1}}^\Lambda(v)
    \right).
\]
Similarly, the \textit{algebraic Frobenius map} is
\[
    \hat\phi_{A_1,\ldots,A_{n+1}}^\Lambda(v)
    =
    \hat m\left(
    \mathcal F_{A_1,\ldots,A_{n+1}}^\Lambda(v)
    \right).
\]
The images of these maps are called the \textit{geometric} and
\textit{algebraic Markov sets} of the semigroup, respectively.
\end{definition}

\begin{example}
Let $\Lambda$ be the classical Farey tessellation, and consider the
semigroup
$
    \langle M_{1,1}^2, M_{1,2}^2 \rangle.
$
Then the corresponding  algebraic and geometric Frobenius maps both coincide the
classical Frobenius map. Hence their images are the classical Markov
numbers.
\end{example}

\begin{remark}
Equip $V\bigl(\mathcal T_\Lambda(\sigma)\bigr)$
with the subspace topology inherited from \(\sigma\). If the Cohn map is
injective, this topology can be transported to its Farey set. If,
in addition, the restriction of \(m\), respectively \(\hat m\), to the
Farey set is injective, then the corresponding geometric, respectively
algebraic, Markov set inherits the transported topology.
\end{remark}

\subsubsection{Semigroups and wug-snake realisations}

Once polyomino wug-tiles have been chosen for the generators of a semigroup,
Proposition~\ref{product of tiles} produces wug-snake realisations of their
products.

\begin{definition}
Let $S=\langle A_1,\ldots,A_n\rangle\subseteq\Mat(d,\mathbb Z)$ be a semigroup. A
\textit{polyomino realisation} of $S$ is a map
\[
    \polyomino:S\to\{\text{bodies of wug-snake graphs}\}
\]
such that:
\begin{itemize}
    \item \(\polyomino(M)\) is a polyomino wug-tile for \(M\), for every
\(M\in S\);
    \item for all $M,N\in S$,
    \[
        \polyomino(MN)=\polyomino(N)\cdot \polyomino(M).
    \]
\end{itemize}
\end{definition}

\begin{remark}
Tiles chosen for the generators always determine a polyomino realisation
of the free semigroup on the symbols \(A_1,\ldots,A_n\). Such a
realisation descends to the matrix semigroup \(S\) only if any two words
representing the same matrix determine the same body. Thus relations in
\(S\) impose compatibility conditions on the chosen generator tiles.
\end{remark}

\subsubsection{Markov semigroups}

Perfect matchings of classical snake graphs provide a combinatorial model
for the classical Markov numbers. By
Corollary~\ref{alg-geo-Markov-det-wug}, wug-snake graphs give an analogous
model for algebraic and geometric Markov numbers of matrices. This raises a
natural question: \textit{when do the algebraic and geometric Markov numbers
coincide?}

\vspace{2mm}

In~\cite{KvS2020}, the authors initiated the study of Markov semigroups of
reduced matrices generated by pairs of matrices. In particular, for positive
integers $a<b$, the semigroup $\langle M_{1,a}^2, M_{1,b}^2\rangle$
is Markov.

\begin{definition}
A matrix \(M\in \Mat(n,\z)\) is called \emph{Markov} if $m(M)=\hat m(M).$ A semigroup \(S \subset \Mat(n,\z)\) is \emph{Markov} if every matrix in its Farey set is Markov; equivalently, if $m(M)=\hat m(M)$ for every \(M\) in the Farey set of \(S\).
\end{definition}

This leads to the following open problems.

\begin{problem}
Study Markov semigroups of reduced matrices with three or more generators.
\end{problem}

\begin{problem}
Study Markov semigroups in higher dimensions.
\end{problem}

\subsection{Examples related to classical Markov numbers}
\label{Several semigroups for classical Markov numbers}

Let us start with four different wug-snake graph representations of the classical Markov numbers: reduced, transpose, Cohn, and $3\times 3$ representations. The first three are represented by new wug-snake graph constructions, and the fourth one corresponds to classical domino snake graphs.

\vspace{2mm}

Recall that the first few Markov numbers are
\[
    1,2,5,13,29,34,89,169,194,233,433,610,985,1325,1597,\ldots.
\]
\subsubsection{Reduced $2\times2$ matrices}

Consider the semigroup generated by 
$ R_1=M_{1,1}^2$ and $R_2=M_{1,2}^2$.
We use the canonical head representing the initial sequence $(0,1)$:
\[
    \includegraphics[height=1.5cm]{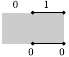}
\]
The numbers above/below the graph indicate the weights of the corresponding
horizontal/vertical edges.
The polyomino wug-tiles for the two generators are
\[
\polyomino(R_1)=
\begin{array}{c}
\includegraphics[height=1.5cm]{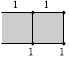}
\end{array},
\qquad
\polyomino(R_2)=
\begin{array}{c}
\includegraphics[height=1.5cm]{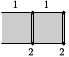}
\end{array}.
\]
In the resulting wug-snake graph, the corresponding Markov number appears
in the second-to-last square.

\begin{proposition}\label{prop:classical-reduced-model}
Let $\Lambda$ be the classical Farey tessellation. For the generating
pair $\langle R_1,R_2\rangle$ defined above, the geometric Frobenius map satisfies
\[
    \phi^\Lambda_{R_1,R_2}(v)=m_v
    \qquad
    \text{for all } v\in\q_{0,1},
\]
where $m_v$ is the classical Markov number with Farey index $v$.
\end{proposition}
\begin{proof}
This is the classical reduced-matrix realisation of the Frobenius map.
\end{proof}
Let us illustrate the above with the following example.

\begin{example}
Consider the reduced matrix
\[
    R=R_2R_1R_2R_1R_1
    =
    \begin{pmatrix}
    119 & 194\\
    284 & 463
    \end{pmatrix}.
\]
The geometric and algebraic Markov numbers of $R$ are both equal to 194. The corresponding wug-snake
graph is
\[
    \includegraphics[height=1.6cm]{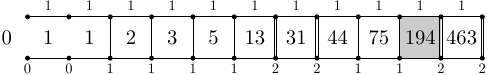}
\]
The numbers inside the squares are the entries of the perfect matching
sequence up to the corresponding square. The Markov number $194$ appears
in the second-to-last square.
%the PLLS-sequence
%\[
%    (1,1,1,1,2,2,1,1,2,2).
%\]
\end{example}

\subsubsection{Transposed matrices}

For \(v\in\mathbb Q_{0,1}\), let $R_v=\mathcal F_{R_1,R_2}^{\Lambda}(v)$ be the reduced matrix associated with \(v\), and set
\[
    S_v=R_v^T.
\]
If $R_v=R_{i_k}\cdots R_{i_1},$ then, since \(R_1\) and \(R_2\) are symmetric, $S_v
    =
    R_{i_1}\cdots R_{i_k}.$
Thus transposition corresponds to reversing the companion word.

Since the geometric Markov number is invariant under transposition, $m(S_v)=m(R_v)=m_v.$ The minimum for \(S_v\) is attained at the transposed initial vector
\((1,0)^T\), rather than at \((0,1)^T\).

The canonical head representing the initial sequence $(1,0)^T$ is as follows:
\[
\begin{array}{c}
\includegraphics[height=1.25cm]{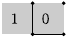}
\end{array}.
\]
All black edges have weight $1$. The polyomino wug-tiles are as above, but
the corresponding Markov number now appears in the last square.

\begin{example}
For the Markov number $194$, the corresponding transposed matrix is
\[
    S=R_1R_1R_2R_1R_2
    =
    \begin{pmatrix}
    119 & 284\\
    194 & 463
    \end{pmatrix}.
\]
The corresponding wug-snake graph is
\[
    \includegraphics[height=1.6cm]{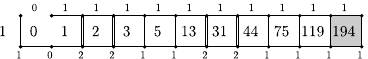}
\]
\end{example}

\subsubsection{Cohn wug-snake graph}
\label{Cohn wug-snake graph}
The elementary Cohn matrices admit the decompositions
\[
    P_1=
    \begin{pmatrix}
    1 & 1\\
    1 & 2
    \end{pmatrix}
    =
    M_{1,1}^2,
    \qquad
    P_2=
    \begin{pmatrix}
    3 & 2\\
    4 & 3
    \end{pmatrix}
    =
    (M_{-1,2}M_{1,1})^2.
\]
In this realisation, some edge weights are negative.

We use the canonical head representing $(0,1)$:
\[
    \includegraphics[height=1.8cm]{figures/snake194-6.pdf}
\]
The polyomino wug-tiles for $P_1$ and $P_2$ are
\[
\polyomino(P_1)=
\begin{array}{c}
\includegraphics[height=1.8cm]{figures/snake194-7.pdf}
\end{array},
\qquad
\polyomino(P_2)=
\begin{array}{c}
\includegraphics[height=1.8cm]{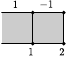}
\end{array}.
\]
The tile for $P_1$ is the same as above. The tile for $P_2$ has last
vertical edge of weight $2$, upper-right horizontal edge of weight $-1$,
and all other displayed edges of weight $1$.

\begin{proposition}
Let $\Lambda$ be the classical Farey tessellation. For the generating
pair $\langle P_1,P_2 \rangle$, the geometric Frobenius map satisfies
\[
    \phi^\Lambda_{P_1,P_2}(v)=m_v
    \qquad
    \text{for all } v\in\q_{0,1}.
\]
\end{proposition}

\begin{proof}
This follows from the preceding transposed model. The matrices here are
conjugate to the reduced matrices used above, and the multiplication order
matches the convention for Cohn's right action. The corresponding matching
sequence is obtained by changing the initial head from $(0,1)$ to the
appropriate transposed initial vector.
\end{proof}

\begin{example}
The word corresponding to the Markov number $194$ is 
$
    P_1P_1P_2P_1P_2.
$
Starting with the vector $(0,1)$, we obtain the wug-snake graph
\[
    \includegraphics[height=1.6cm]{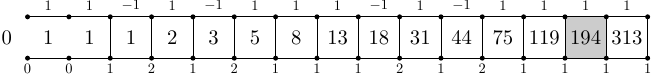}
\]
In the figure, ordinary, double, and negative edges are all represented by
single line segments; their weights are indicated separately.
\end{example}

\subsubsection{Classical snake graphs as $3\times3$ systems}
\label{Classical snake graphs as 3x3 systems}

Classical snake graphs satisfy recurrences of order $2$. In our
wug-snake framework, it is often convenient to represent them using
$3\times3$ matrices. We use the generators
\[
    M_1=M_{0,1,1}^2,
    \qquad
    M_2=\bigl(M_{1,0,1}M_{0,1,1}\bigr)^2.
\]
The canonical head representing $(0,0,1)$ is
\[
    \includegraphics[height=1.25cm]{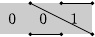}
\]
The polyomino wug-tiles for $M_1$ and $M_2$ are
\[
\polyomino(M_1)=
\begin{array}{c}
\includegraphics[height=1.25cm]{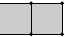}
\end{array},
\qquad
\polyomino(M_2)=
\begin{array}{c}
\includegraphics[height=1.25cm]{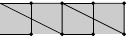}
\end{array}.
\]
All black edges have weight $1$.

\vspace{2mm}

The wug-snake graph corresponding to the word
$
    M_1M_1M_2M_1M_2
$
is
\[
    \includegraphics[height=1.00cm]{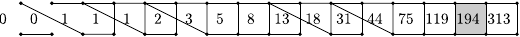}
\]
An isometric classical snake graph is shown below:
\[
    \includegraphics[height=5.3cm]{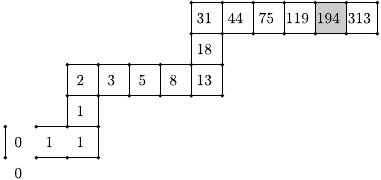}
\]
In this model, the classical snake graph starts with the value $2$, after
removing the first five entries of the associated linear recurrence system.

\begin{remark}
In this $3\times3$ realisation, we recover classical Markov numbers as
entries of matrices rather than as wug-snake determinants. For example,
\[
M_1M_1M_2M_1M_2=
\begin{pmatrix}
0&165&119\\
0&269&194\\
0&434&313
\end{pmatrix}.
\]
The Markov number $194$ appears as the middle entry of the last column.
However, the geometric Markov number of this $3\times3$ matrix is $0$. Indeed, the first column of this matrix is zero. Hence \(Me_1=0\), and therefore $f_M(e_1)=\det(e_1,Me_1,M^2e_1)=0.$ Consequently, \(m(M)=0\).
\end{remark}

\subsection{Semigroups with PLLS-sequences $(a,a)$ and $(b,b)$}
\label{section-aa-bb}

We now consider semigroups generated by two $2\times2$ matrices whose
PLLS-sequences are $(a,a)$ and $(b,b)$, respectively, with $a<b$, and
which are simultaneously reduced in the same basis. In this setting the
algebraic and geometric Markov numbers coincide. The first values are
\[
    a,\quad b,\quad a^2b+a+b, \quad
    a^4b+a^3+3a^2b+2a+b,
\quad  a^2b^3+2a^2b+ab^2+b^3+a+2b,\ldots
\]
For $a=1$ and $b=2$, this gives the classical Markov numbers
$
    1,2,5,13,29,\ldots
$
\begin{remark}
The standard Cohn matrices are conjugate to $M_{1,1}^2$ and
$M_{1,2}^2$:
\[
\begin{pmatrix}
1 & 1\\
1 & 2
\end{pmatrix}
=
M_{1,1}M_{1,1}^2M_{1,1}^{-1},
\qquad
\begin{pmatrix}
3 & 2\\
4 & 3
\end{pmatrix}
=
M_{1,1}M_{1,2}^2M_{1,1}^{-1}.
\]
This explains why perfect matchings of wug-snake graphs naturally describe
the semigroup generated by Cohn matrices, even though one of the generators
is not reduced in the original basis. It also explains why the Markov
number $1$ does not appear in the usual snake graph model.
\end{remark}

Motivated by this observation, we consider
\[
    M_1=M_{1,a}^2,
    \qquad
    M_2=\bigl(M_{1,a}M_{1,b}M_{1,a}^{-1}\bigr)^2.
\]

\begin{proposition}\label{aa-bb}
The bodies $($in white$)$ in the following wug-snake graphs represent $M_1$ and $M_2$, respectively:
\[
\begin{array}{c}
\includegraphics[height=1.5cm]{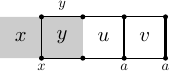}
\end{array}
\qquad\quad
\begin{array}{c}
\includegraphics[height=3.2cm]{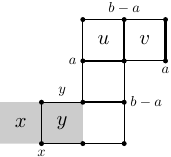}
\end{array}
\]
That is,
\[
    \begin{pmatrix}
    u\\
    v
    \end{pmatrix}
    =
    M_i
    \begin{pmatrix}
    x\\
    y
    \end{pmatrix},
    \qquad i=1,2.
\]
\end{proposition}
\begin{proof}
This follows by direct computation of the corresponding snake operators.
\end{proof}

\begin{remark}
The wug-snake graphs generated by these two tiles represent the elements in
the image of the Cohn map for the semigroup generated by $M_1$ and
$M_2$. They coincide with classical domino snake graphs whose edge
weights are $a$, $b-a$, or $1$.
\end{remark}

\begin{proposition}
For every matrix \(M\) in the Farey set associated with
\((M_1,M_2)\) and the classical Farey subdivision scheme, one has $m(M)=\hat m(M).$ Thus the semigroup generated by \(M_1\) and \(M_2\) is Markov with
respect to this generating pair and subdivision scheme.
\end{proposition}

\begin{proof}
This follows from the general theory developed in~\cite{KvS2020}.
\end{proof}

\begin{remark}
As in the classical case, the corresponding wug-snake graphs can also be
generated by $3\times3$ matrices. Define
\[
    \widetilde M_1=A^2,
    \qquad
    \widetilde M_2=(CB)^2,
\]
where
\[
A=
\begin{pmatrix}
0&1&0\\
0&0&1\\
0&1&a
\end{pmatrix},
\qquad
B=
\begin{pmatrix}
0&1&0\\
0&0&1\\
0&b-a&1
\end{pmatrix},
\qquad
C=
\begin{pmatrix}
0&1&0\\
0&0&1\\
1&0&a
\end{pmatrix}.
\]
The wug-snake bodies associated with
\(\widetilde M_1,\widetilde M_2\) produce the same weighted snake-graph
patterns as those associated with the preceding \(2\times2\) model,
after replacing heads of length \(2\) by heads of length \(3\).
\end{remark}

\subsection{A three-generator example}
\label{Further example}

We conclude with a three-generator model whose wug-snake bodies admit realisations in \(\mathbb R^3\).

Consider the following transposed Frobenius companion matrices:
\[
A=
\begin{pmatrix}
0 & 1 & 0 & 0\\
0 & 0 & 1 & 0\\
0 & 0 & 0 & 1\\
0 & 0 & 1 & 1
\end{pmatrix},
\qquad
B=
\begin{pmatrix}
0 & 1 & 0 & 0\\
0 & 0 & 1 & 0\\
0 & 0 & 0 & 1\\
0 & 1 & 0 & 1
\end{pmatrix},
\qquad
C=
\begin{pmatrix}
0 & 1 & 0 & 0\\
0 & 0 & 1 & 0\\
0 & 0 & 0 & 1\\
1 & 0 & 0 & 1
\end{pmatrix}.
\]
We study the semigroup generated by
\[
    M_1=A,
    \qquad
    M_2=BA,
    \qquad
    M_3=CBA.
\]
The following figures show polyomino wug-tiles for $M_1,M_2,M_3$ together
with their embeddings into $\r^3$:
\[
M_1:
\begin{array}{c}
\includegraphics[height=1.3cm]{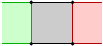}
\end{array}
\qquad\longrightarrow\qquad
\begin{array}{c}
\includegraphics[height=.75cm]{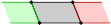}
\end{array}
\]
\[
M_2:
\begin{array}{c}
\includegraphics[height=1.3cm]{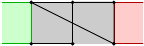}
\end{array}
\qquad\longrightarrow\qquad
\begin{array}{c}
\includegraphics[height=2cm]{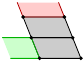}
\end{array}
\]
\[
M_3:
\begin{array}{c}
\includegraphics[height=2cm]{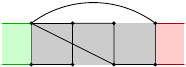}
\end{array}
\qquad\longrightarrow\qquad
\begin{array}{c}
\includegraphics[height=3cm]{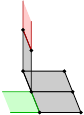}
\end{array}
\]
The grey squares represent $\polyomino(M_1)$, $\polyomino(M_2)$, and $\polyomino(M_3)$, respectively. The green and red directions indicate how a
tile is connected to the preceding and following tiles. In some cases, the
tiles must be rotated as rigid objects in space before being attached.

The generator $M_1$ changes neither direction nor plane. The generator
$M_2$ changes direction but remains in the same plane, while $M_3$
changes the plane itself.

\begin{remark}
This construction gives natural embeddings of certain wug-snake graphs into
$\r^3$.
\end{remark}

\begin{example}
Consider the product
\[
    M_1M_3M_2M_1M_3M_2M_1
    =
    ACBABA^2CBABA^2.
\]
The corresponding concatenated body is shown below:
\[
\begin{array}{c}
\includegraphics[height=1.6cm]{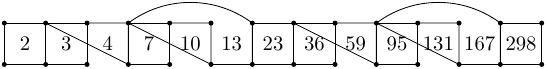}
\end{array}
\]
Here we do not specify the head. A realisation of this body in \(\mathbb Z^3\) is:
\[
\begin{array}{c}
\includegraphics[height=5cm]{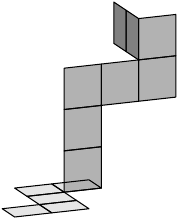}
\end{array}
\]
\end{example}

\begin{remark}
The geometric pattern of these bodies is motivated by the
\(2\times2\) matrices
\[
    M_{1,1}M_{1,i}M_{1,1}^{-1}.
\]
The cases $i=1,2$ appear in the classical construction of domino snake
graphs. The case $i=3$ leads to the three-dimensional construction
discussed here. Similar twists can be defined in higher dimensions.
\end{remark}

Finally, we would like to mention that there are several ways to embed polyomino constructions in
three-dimensional space. For instance one can consider the matrix
\[
    \widehat M_3=CA^2
\]
instead of $M_3$. The corresponding embedding is
\[
\widehat M_3:
\begin{array}{c}
\includegraphics[height=2cm]{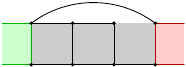}
\end{array}
\qquad\longrightarrow\qquad
\begin{array}{c}
\includegraphics[height=3cm]{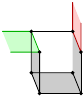}
\end{array}.
\]

\section{Lattice realisations of wug-snake graphs}
\label{Lattice realisation of wug-snake graphs}

In this section we discuss lattice realisations of wug-snake graphs. We begin
with the classical two-dimensional case, where domino snake graphs are
embedded in the half-integer grid and their slopes recover Farey coordinates.
We then adapt this construction to the semigroups with PLLS-sequences $(a,a)$ and $(b,b)$ studied in Subsection~\ref{section-aa-bb}. After
that, we discuss possible extensions to semigroups with three or more
generators. Finally, we formulate several problems concerning slowly
increasing sequences of cubes and their additive structure.

\subsection{Slopes for classical domino snake graphs}
\label{Slopes for classical dominoe snake graphs}

We start with the classical two-dimensional construction. The goal is to
associate to a domino snake graph a lattice embedding and a rational slope.
This will serve as the model for the higher-dimensional constructions below.

Let $M$ be a Cohn matrix. We denote by $\nu(M)$ an embedding of the
associated domino snake graph into the half-integer coordinate grid. The
slope of this embedding is defined by the translation vector from the first
square to the last square of $\nu(M)$. We denote its tangent by
$\tan_\nu(M)$.

\subsubsection*{Embeddings of the generators}

Recall that H.~Cohn used the following matrices as generators:
\[
M_1=
\begin{pmatrix}
1 & 1\\
1 & 2
\end{pmatrix},
\qquad
M_2=
\begin{pmatrix}
3 & 2\\
4 & 3
\end{pmatrix}.
\]
We choose the following embeddings of their associated domino snake graphs:
\[
\nu(M_1)=
\begin{array}{c}
\includegraphics[height=1.5cm]{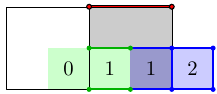}
\end{array},
\qquad
\nu(M_2)=
\begin{array}{c}
\includegraphics[height=3cm]{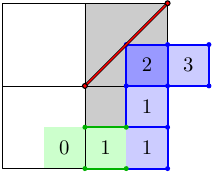}
\end{array}.
\]
Note that in the classical case of snake graphs, the very last small square is not taken into account. So these realisation represents the Markov numbers $1$ and $2$ respectively.

\vspace{2mm}

The head here is represented by the sequence $(0,1)$, and we
place the square corresponding to $0$ in the unit coordinate square adjacent to the bottom right vertex of the unit square. In the figures, the head is shown in
green and the body of the corresponding snake graph is shown in blue. The
red vector indicates the translation from the first square to the last
square. 
Thus
\[
\tan_\nu(M_1)=\frac{0}{1},
\qquad
\tan_\nu(M_2)=\frac{1}{1}.
\]

\subsubsection*{Products and Farey addition of slopes}

Suppose that embeddings have been constructed for two matrices $N_1$ and
$N_2$, with corresponding graphs
$H_1B_1$
and
$H_2B_2$.
The domino snake graph corresponding to the product $N_1N_2$ is obtained
by attaching the body $B_1$ after $H_2B_2$:
\[
H_2B_2\oplus H_1B_1=H_2B_2B_1.
\]
Equivalently, the body of the first factor is attached to the graph of the
second factor, in accordance with the convention for matrix-vector
multiplication.

\vspace{2mm}

We define the embedding $\nu(N_1N_2)$ by starting with the embedding of
$H_2B_2$ and then attaching the squares corresponding to $B_1$, beginning
at the square immediately adjacent to the last square of $\nu(N_2)$.

\begin{definition}[Iterative rule for tangents]
Assume that
\[
\tan_\nu(N_1)=\frac{y_1}{x_1},
\qquad
\tan_\nu(N_2)=\frac{y_2}{x_2},
\]
where $(x_i,y_i)$ are pairs of relatively prime nonnegative integers.
Then we define
\[
\tan_\nu(N_1N_2)
=
\frac{y_1}{x_1}\oplus\frac{y_2}{x_2}
=
\frac{y_1+y_2}{x_1+x_2}. 
\]
\end{definition}

The following example illustrates the construction.

\begin{example}
Let $N_1=M_1M_2$ and $N_2=M_2$, so that $m(N_1)=5$ and $m(N_2)=2$, and their embeddings have slopes
\[
\begin{array}{c}
\nu(N_2)=
\begin{array}{c}
\includegraphics[height=3cm]{figures/snake-add-2.pdf}
\end{array}
\\[1mm]
\displaystyle
\tan_\nu(N_2)=\frac{1}{1}
\end{array}
\qquad
\begin{array}{c}
\nu(N_1)=
\begin{array}{c}
\includegraphics[height=3cm]{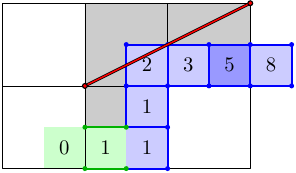}
\end{array}
\\[1mm]
\displaystyle
\tan_\nu(N_1)=\frac{1}{2}.
\end{array}
\]
Now consider
$N_1N_2=M_1M_2^2$,
with $m(N_1N_2) = 29$. The corresponding embedding is
\[
\begin{array}{c}
\nu(N_1N_2)=
\begin{array}{c}
\includegraphics[height=4.5cm]{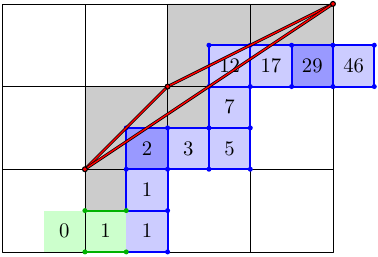}
\end{array}
\\[1mm]
\displaystyle
\tan_\nu(N_1N_2)
=
\frac{1}{1}\oplus\frac{1}{2}
=
\frac{2}{3}.
\end{array}
\]
The red triangle is bounded by an edge of the Farey tessellation. Hence it
contains no interior lattice points, and the corresponding ray intersects
precisely the union of the two snake graphs.

\vspace{2mm}

Recall that in order to obtain the classical snake graph, one should remove the first 4 squares and the last 1 square from the last picture. Then we get:
$$
\begin{array}{c}
\includegraphics[height=3cm]{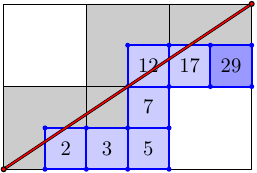}
\end{array}
$$
\end{example}

\begin{remark}
In the above convention, the matrix with the larger tangent is placed first
in the geometric concatenation.
\end{remark}

The classical relation between slopes and Markov numbers can now be
summarised as follows.

\begin{proposition}
Let $\mathcal C(v)=\mathcal F^\Lambda_{M_1,M_2}(v)$ be the Cohn matrix associated with a vertex
\(v=p/q\in\mathbb Q_{0,1}\) of the classical Farey tessellation. Then,
\[
    \tan_\nu(\mathcal C(v))=v
    \qquad\text{and}\qquad
    m(\mathcal C(v))=m_v.
\]
\end{proposition}
\begin{remark}
The map $\tan_\nu(M)\longmapsto m(M)$ coincides with the classical Frobenius map.
\end{remark}
For further details on this classical picture, see~\cite{Schiffler2024}.

\subsection{The case of PLLS-sequences $(a,a)$ and $(b,b)$}
\label{Case of PLLS-sequences a,a,b,b}

We now adapt the preceding construction to the semigroups considered in
Subsection~\ref{section-aa-bb}. Let
\[
M_1=M_{1,a}^2,
\qquad
M_2=\bigl(M_{1,a}M_{1,b}M_{1,a}^{-1}\bigr)^2.
\]
The reduced
representatives of \(M_1\) and \(M_2\) have PLLS-sequences \((a,a)\) and \((b,b)\). We define embeddings of the generators by
\[
\begin{array}{c}
\nu_{[a,a,b,b]}(M_1)=
\begin{array}{c}
\includegraphics[height=1.5cm]{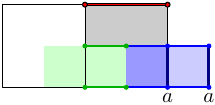}
\end{array}
\\[1mm]
\displaystyle
\tan_{\nu_{[a,a,b,b]}}(M_1)=\frac{0}{1}
\end{array}
\qquad
\begin{array}{c}
\nu_{[a,a,b,b]}(M_2)=
\begin{array}{c}
\includegraphics[height=3cm]{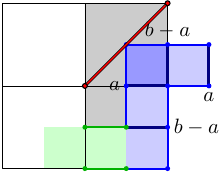}
\end{array}
\\[1mm]
\displaystyle
\tan_{\nu_{[a,a,b,b]}}(M_2)=\frac{1}{1}.
\end{array}
\]
\begin{theorem}
Let \(\Lambda\) be the classical Farey tessellation, and, for
\(v\in\mathbb Q_{0,1}\), set
\[
    A_v=\mathcal F^\Lambda_{M_1,M_2}(v).
\]
If \(v=p/q\) is written in lowest terms, then the embedding
\(\nu_{[a,a,b,b]}(A_v)\) has primitive displacement vector \((q,p)\).
In particular,
\[
    \tan_{\nu_{[a,a,b,b]}}(A_v)=v.
\]
Consequently, the map
\[
    v\longmapsto \hat m(A_v)
\]
is precisely the algebraic Frobenius map
\(\hat\phi^\Lambda_{M_1,M_2}\). Thus the rational directions parametrize,
not necessarily injectively, the algebraic Markov set of
\(\langle M_1,M_2\rangle\) with respect to \(\Lambda\).
\end{theorem}

\begin{proof}
The proof is combinatorially identical to the classical case. The same
Farey concatenation of slopes produces the same pattern of generator words,
with the entries $1$ and $2$ in the classical model replaced by $a$
and $b$, respectively.
\end{proof}

\begin{remark}
One may similarly consider pairs of PLLS-sequences of different lengths and
with different entries. In that case, the elementary polyomino blocks are
replaced by more general graph textures inside coordinate boxes. The theory
of the slope map $\tan_\nu$ remains unchanged, since it depends only on
the arrangement of the boxes and not on the internal texture.
\end{remark}

We end this subsection by recalling a counterexample to the uniqueness
conjecture for $((4,4),(11,11))$-Markov numbers from~\cite{KvS2020}.

\begin{example}
Consider the semigroup generated by
\[
M_1=M_{1,4}^2,
\qquad
M_2=\bigl(M_{1,4}M_{1,11}M_{1,4}^{-1}\bigr)^2.
\]
Consider two elements of this semi-group with Farey coordinates $4/5$ and $1/7$,
respectively:
$$
M_2^4M_1=
\left(\begin{array}{cc}
83712817 & 355318099 
\\
 928389367 & 3940538102 
\end{array}\right),
M_2M_1^6=
\left(\begin{array}{cc}
83879225 & 355318099 
\\
 930136651 & 3940122082
\end{array}\right).
$$
The algebraic and geometric Markov numbers are both equal to
$
355318099. 
$

Hence the uniqueness conjecture in this semigroup does not hold.
For further examples, see~\cite{KvS2020,Matty2019}.
\end{example}

\subsection{The case of three or more generators}
\label{Case of three and more generators}

We now briefly indicate how the preceding picture can be extended to
semigroups with three or more generators. Let
$
M_1,\ldots,M_n
$
be generators of a semigroup, with $n\geq 3$.

\vspace{2mm}

For each generator $M_i$, we choose an embedding $\nu(M_i)$ into
$\r^n$. We formally assign the tangent vector
$
\tan_\nu(M_i)=(0,\ldots,0,1,\ldots,1),
$
with $i$ terminal entries equal to $1$.

\vspace{2mm}

Next, choose a multidimensional Farey tessellation, for instance one of the
tessellations from Example~\ref{Farey tessellations examples}. The rule for
multiplying matrices is accompanied by a rule for concatenating the
corresponding embedded wug-snake graphs. The desired compatibility condition
is
\[
\tan_\nu(AB)=\tan_\nu(A)\oplus\tan_\nu(B),
\]
where $\oplus$ denotes the multidimensional Farey sum of rational vectors
introduced in Subsection~\ref{Farey addition of rational vectors}.

\vspace{2mm}

Unlike the two-generator case, there are many possible choices of initial
embeddings, Farey tessellations, and concatenation rules. These choices may
lead to different theories of higher-dimensional wug-snake graphs. We now
illustrate this phenomenon with an example.

\subsubsection{A three-generator example}
\label{Example of a semigroup related to such snake graphs}

Consider the semigroup generated by the following three matrices:
\[
\begin{aligned}
M_1&=
\begin{pmatrix}
1 & 1\\
1 & 2
\end{pmatrix}
=
M_{1,1}M_{1,1}^2M_{1,1}^{-1},\\[2mm]
M_2&=
\begin{pmatrix}
3 & 2\\
4 & 3
\end{pmatrix}
=
M_{1,1}M_{1,2}^2M_{1,1}^{-1},\\[2mm]
M_3&=
\begin{pmatrix}
14 & 5\\
25 & 9
\end{pmatrix}
=
M_{1,1}\bigl(M_{1,1}M_{1,3}\bigr)^2M_{1,1}^{-1}.
\end{aligned}
\]
The first two matrices are the standard Cohn matrices. The third matrix has
PLLS-sequence $(3,1,3,1)$.

We consider the following embeddings of their polyomino wug-snake graphs into
$\r^3$, obtained by taking double iterates of the constructions introduced
in Subsection~\ref{Further example}:
\[
\begin{array}{l}
\nu(M_1)=
\begin{array}{c}
\includegraphics[height=3.3cm]{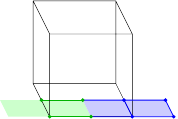}
\end{array}
\\
\nu(M_2)=
\begin{array}{c}
\includegraphics[height=4cm]{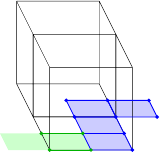}
\end{array}
\end{array}
\qquad
\nu(M_3)=
\begin{array}{c}
\includegraphics[height=6.5cm]{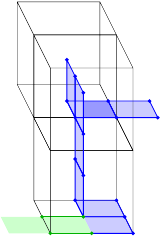}.
\end{array}
\]
These embeddings follow a double subdivision of the original cubes. Their
tangent vectors are
\[
\tan_\nu(M_1)=(0:0:1),
\qquad
\tan_\nu(M_2)=(0:1:1),
\qquad
\tan_\nu(M_3)=(1:1:1).
\]

\begin{example}\label{example-3d}
Consider the product
$
M_3M_2=
\begin{pmatrix}
62 & 43\\
111 & 77
\end{pmatrix}.
$
The corresponding wug-snake graph is
\[
\begin{array}{c}
\includegraphics[height=1.5cm]{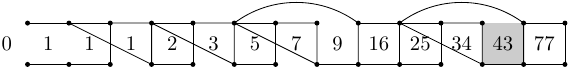}
\end{array}.
\]
Its realisation in $\r^3$ is shown in the figure below on the left. In the
classical reduction, one removes the first four squares and the final square
on the right; this is shown in the figure below on the right:
\[
\begin{array}{c}
\includegraphics[height=5cm]{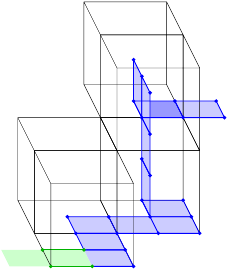}
\end{array}
\qquad\qquad
\begin{array}{c}
\includegraphics[height=4cm]{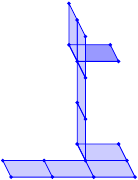}
\end{array}.
\]
After removing the first four squares and the last square, and after
untwisting the first remaining square, the wug-snake graph becomes
\[
\begin{array}{c}
\includegraphics[height=1.5cm]{figures/snake3d-3.pdf}
\end{array}.
\]
In this representation, the factor $M_2$ is hidden in the geometry of the
graph and is no longer directly visible from the final planar form.
\end{example}
\begin{remark}
In the classical theory, the initial entries $0$ and $1$ are removed from
the matching sequence. As a result, the Markov number $1$ is not represented
by the usual snake graphs. In the wug-snake setting, however, it is naturally
represented by the matrix $M_1$, corresponding to the wug-snake graph
\[
\begin{array}{c}
\includegraphics[height=1cm]{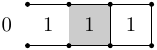}
\end{array}.
\]
\end{remark}

Let us now return to the construction of Markov numbers for this semigroup.
We use the pairwise Farey summation described in
Example~\ref{Farey tessellations examples}. The corresponding embedded
wug-snake graphs are constructed by attaching snake bodies in the direction
of the first coordinate vector. We choose heads corresponding to the initial
vector $(x,y)=(0,1)$.

The Markov numbers obtained after the first two subdivision steps are shown
below:

\[
\begin{array}{c}
\includegraphics[height=5cm]{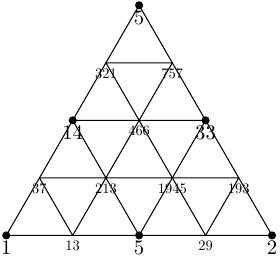}
\end{array}
\]
The left vertex of the triangle corresponds to $M_1$, the right vertex to $M_2$, and the top vertex to $M_3$.

\vspace{2mm}

Let us conclude this subsection with the following structural proposition.

\begin{proposition}\label{Cohn-extended-snake-prop}
For $k\geq 2$, consider the matrix
\[
N_k
=
M_{1,1}M_{1,k}M_{1,1}^{-1}
=
\begin{pmatrix}
k-1 & 1\\
k & 1
\end{pmatrix}.
\]
Then $N_k$ admits a natural family of polyomino wug-tile realisations in
$\r^k$. These realisations are constructed recursively as follows.
\begin{itemize}
    \item[\textup{Step 1.}] Choose a coordinate square in the
    $(x_1,x_2)$-plane and fix one of its vertices, denoted by $V$.

    \item[\textup{Step 2.}] Add a second coordinate square in the
    $(x_1,x_2)$-plane, distinct from the first one, sharing an edge with it
    and also containing the vertex $V$.

    \item[\textup{Step $i$.}] Suppose that the first $i-1$ squares have
    been constructed. The $i$-th square is chosen in the
    $(x_{i-1},x_i)$-coordinate plane so that it contains the vertex $V$ and
    shares an edge with the previously constructed square.
\end{itemize}
\end{proposition}

\begin{proof}
The statement follows by a straightforward induction on $i$. At each step,
the newly added coordinate square realizes the next elementary transition in
the corresponding wug-snake body, and the resulting snake operator is
multiplied by the required factor. After $k$ steps, the induced operator is
precisely
\[
N_k=
\begin{pmatrix}
k-1 & 1\\
k & 1
\end{pmatrix}.
\]
\end{proof}

\begin{corollary}
Once a realisation of each $N_k$ in $\r^k$ is fixed, we obtain natural
polyomino realisations in $\r^\infty$ for all elements of the semigroup
\[
\left\{
M_{1,1}NM_{1,1}^{-1}
\,\middle|\,
N\in\LC_{\geq0}(2,\z)
\right\}.
\]
\end{corollary}

\begin{proof}
The polyomino realisations of the generators are provided by
Proposition~\ref{Cohn-extended-snake-prop}. A realisation of a product is
then obtained by concatenating the corresponding polyomino wug-tiles.

There are several possible rules for choosing the ambient embedding of such
products. One simple rule is to increase the ambient dimension whenever a
new generator is attached. Another is to keep the direction of the tail of
the snake, which is a standard basis vector $e_i$, as small as possible at
each step.

Since the semigroup of reduced matrices is freely generated, each element
has a unique word in the generators. Therefore, once an embedding rule is
fixed, each element of the conjugated semigroup has a well-defined associated
polyomino realisation.
\end{proof}

\subsection{Slowly increasing sequences of cubes}
\label{Problem on addition of slowly increasing sequences of cubes}

In the previous subsection, wug-snake graphs were used as textures on
configurations of cubes. We now formulate this idea more systematically.
The wug-snake graph should be viewed as an additional combinatorial texture
placed on a sequence of cubes; the underlying lattice configuration is
defined independently.

\subsubsection{Slowly increasing sequences of cubes and their wug-snake graphs}

We begin with a higher-dimensional analogue of a sequence of adjacent
squares.

\begin{definition}
A finite or infinite sequence of points $(v_1,\ldots,v_n)$ in $\z^d$ is called \textit{slowly increasing} if $v_{j+1}-v_j$ is a standard coordinate vector for every $j$. Let $\#(j)$ denote the
index of this coordinate vector, and set $\#(0)=1$. We define
\[
c(j)=\#(j)-\#(j-1)\pmod d
\]
and call $c(j)$ the \textit{cyclic shift at position $j$}.
\end{definition}

Recall that $\r^d$ has the standard lexicographic order on vertices. For
example, in $\r^3$ the point $(1,2,3)$ precedes $(1,3,2)$.

\begin{definition}
A \textit{slowly increasing sequence of cubes} is the union of all unit
lattice cubes whose lexicographically smallest vertices form a slowly
increasing sequence.
\end{definition}

We now associate wug-snake graphs to such cube sequences algebraically.

\begin{definition}
Let $G$ be a semigroup generated by $A_1,\ldots,A_d$. Let
\[
S=(v_1,\ldots,v_n)
\]
be a slowly increasing sequence in $\z_{\geq0}^d$ starting at the origin.
The element
\[
A(S)=
\prod_{j=1}^{n} A_{c(n-j+1)+1}
\]
is called the \textit{representative} of $S$. The corresponding
wug-snake graph is denoted by $\wug(S)$.
\end{definition}

This definition assigns a wug-snake texture to a sequence of cubes, just as
classical domino snake graphs assign textures to increasing sequences of
squares in the two-dimensional case.

\begin{example}\label{cubes-1}
Consider the following slowly increasing sequence of cubes:
\[
\begin{array}{c}
\includegraphics[height=5cm]{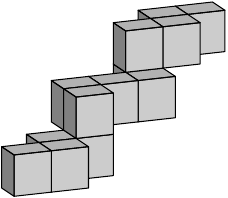}
\end{array}
\]
It corresponds to the semigroup element
\[
A_2A_2A_2A_2A_2A_1A_3A_3A_3A_3A_2A_1
=
A_2^5A_1A_3^4A_2A_1.
\]
Here the matrices are written in reverse order, following the convention for
matrix-vector multiplication. For the semigroup considered in
Subsection~\ref{Example of a semigroup related to such snake graphs}, the
corresponding number of perfect matchings is
\[
36313494507.
\]
\end{example}

The preceding construction gives a natural map from certain elements of the
semigroup generated by $A_1,\ldots,A_d$ to lattice configurations in
$\z_{\geq0}^d$.

\subsubsection{Slowly increasing sequences of cubes and slopes}

Conversely, one can associate a slowly increasing sequence of cubes to many
integer vectors. We first fix a tie-breaking convention.

\vspace{2mm}

Consider a unit cube whose faces are parallel to the coordinate hyperplanes. Then its vertices are {\it lexicographically ordered} by their the sequences if their coordinates. This order 
naturally induces an ordering on the space of the $k$-faces of the cube. We call this ordering the {\it lexicographic ordering}.

\vspace{2mm}

Now we can define the sequence associated with a vector.
\begin{definition}
Let $v\in\z^d$. Consider the open line segment from the origin to $v$.
Take all unit lattice cubes intersected by this segment. If the segment
passes through a lower-dimensional face shared by several cubes, choose the
cube for which this face is lexicographically smallest. The resulting
slowly increasing sequence of cubes is said to be \textit{associated with}
$v$, and is denoted by $C(v)$.
\end{definition}

\begin{example}
The slowly increasing sequence shown in Example~\ref{cubes-1} is
$C(7,5,3)$.
\end{example}

In dimension two, the concatenation of snake graphs is compatible with
Farey addition of slopes. In higher dimensions, however, the sum of two
slowly increasing sequences associated with vectors need not be slowly
increasing. This motivates the following problems.

\begin{problem}
For which vectors $v_1$ and $v_2$ does a natural concatenation of
$C(v_1)$ and $C(v_2)$ produce $C(v_1+v_2)$?
\end{problem}

Here the word ``concatenation'' is intentionally left flexible. One possible
operation is to attach the sequence $C(v_2)$ to the final hyperface of $C(v_1)$, for instance along the hyperface orthogonal to the first
coordinate vector.

\begin{problem}
Develop an additive theory of slowly increasing sequences of cubes.
\end{problem}

\printbibliography

\end{document}